\documentclass[10pt,reqno]{amsart}
\usepackage{latexsym}
\usepackage{amscd}
\usepackage{amssymb}
\usepackage{graphicx}
\usepackage{epstopdf}
\usepackage{amsmath}
\usepackage{textcomp}

\usepackage{cancel}
\usepackage{tikz}
\usetikzlibrary{decorations.markings}
\tikzstyle{vertex}=[circle, draw, inner sep=0pt, minimum size=6pt]
\newcommand{\vertex}{\node[vertex]}

\usepackage{caption}
\usepackage{subcaption}
\usepackage{amssymb, amscd, amsmath,amsthm,amsfonts,epsfig, enumerate}
\usepackage{epsfig}
\usepackage{fullpage}
\usepackage{listings}
\usepackage{float}
\usepackage{epstopdf}
\usepackage{color}

\def\NZQ{\Bbb}               

\def\ZZ{{\NZQ Z}}

\def\B'c{{\mathcal{B'}}}
\def\U'c{{\mathcal{U'}}}
\def\opn#1#2{\def#1{\operatorname{#2}}} 
\opn\chara{char}
\opn\length{\ell}
\opn\projdim{proj\,dim}
\opn\injdim{inj\,dim}
\opn\ini{in}
\opn\rank{rank}
\opn\depth{depth}
\opn\sdepth{sdepth}
\opn\indmat{indmat}
\opn\cochord{cochord}
\opn\pdim{pdim}
\opn\height{ht}
\opn\embdim{emb\,dim}
\opn\codim{codim}

\opn\Tr{Tr}
\opn\bigrank{big\,rank}
\opn\superheight{superheight}\opn\lcm{lcm}
\opn\trdeg{tr\,deg}%
\opn\reg{reg}
\opn\lreg{lreg}
\opn\set{set}
\opn\supp{Supp}
\opn\shad{Shad}
\opn\div{div}
\opn\Div{Div}
\opn\cl{cl}
\opn\Cl{Cl}
\opn\Spec{Spec}
\opn\Supp{Supp}
\opn\supp{supp}
\opn\Sing{Sing}
\opn\Ass{Ass}
\opn\Min{Min}
\opn\size{size}
\opn\bigsize{bigsize}
\opn\lex{lex}
\opn\Ann{Ann}
\opn\Rad{Rad}
\opn\Soc{Soc}
\opn\Ker{Ker}
\opn\Coker{Coker}
\opn\Im{Im}
\opn\Hom{Hom}
\opn\Tor{Tor}
\opn\Ext{Ext}
\opn\End{End}
\opn\Aut{Aut}
\opn\id{id}

\opn\nat{nat}
\opn\GL{GL}
\opn\SL{SL}
\opn\mod{mod}
\opn\ord{ord}
\opn\aff{aff}
\opn\con{conv}
\opn\relint{relint}
\opn\st{st}
\opn\lk{lk}
\opn\cn{cn}
\opn\core{core}
\opn\vol{vol}
\opn\gr{gr}

\def\pot#1#2{#1[\kern-0.28ex[#2]\kern-0.28ex]}

\opn\dirlim{\underrightarrow{\lim}}
\opn\invlim{\underleftarrow{\lim}}
\let\tensor=\otimes

\def\pnt{{\raise0.5mm\hbox{\large\bf.}}}

\def\Implies{\ifmmode\Longrightarrow \else
	\unskip${}\Longrightarrow{}$\ignorespaces\fi}
\def\implies{\ifmmode\Rightarrow \else
	\unskip${}\Rightarrow{}$\ignorespaces\fi}
\def\iff{\ifmmode\Longleftrightarrow \else
	\unskip${}\Longleftrightarrow{}$\ignorespaces\fi}

\let\:=\colon
\newtheorem{Theorem}{Theorem}[section]
\newtheorem{Lemma}[Theorem]{Lemma}
\newtheorem{Corollary}[Theorem]{Corollary}
\newtheorem{Proposition}[Theorem]{Proposition}
\newtheorem{Remark}[Theorem]{Remark}

\let\epsilon=\varepsilon
\let\phi=\varphi
\let\kappa=\varkappa
\numberwithin{equation}{section}

\title{ Algebraic invariants of edge ideals of cubic circulant graphs}

\author[Bakhtawar Shaukat$^1$]{Bakhtawar Shaukat$^1$}

\author[Muhammad Ishaq$^1,*$]{Muhammad Ishaq$^{1,*}$}

\author[Ahtsham ul Haq$^1$]{Ahtsham ul Haq$^1$}

\author[Zahid Iqbal$^2$]{Zahid Iqbal$^2$}

\begin{document}
\maketitle
\begin{center}
   $^1$School of Natural Sciences, National University of Sciences and Technology Islamabad, Sector H-12, Islamabad Pakistan.\linebreak
     $^2$Department of Mathematics and Statistics, Institute of Southern Punjab, Multan, Pakistan.\linebreak
     $^*$ Corresponding email: \email{ishaq\_maths@yahoo.com}; Tel: +92-51-90855591\end{center}
\begin{abstract}
We obtain the exact values for depth and projective dimension and lower bounds for Stanley depth of the quotient rings of the edge ideals associated with all cubic circulant graphs.

\end{abstract}
\section{Introduction}

Let $K$ be a field and $S=K[x_1, \dots , x_q]$ be the polynomial ring over $K$ with standard grading. Let $\mathcal{M}$ be a finitely generated graded $S$-module. Suppose that $\mathcal{M}$ admits the following minimal free resolution:
	$$0\longrightarrow\ \bigoplus_{j\in \ZZ} S(-j)^{\beta_{r,j}(\mathcal{M})} \longrightarrow \bigoplus_{j\in \ZZ}  S(-j)^{\beta_{r-1,j}(\mathcal{M})}
	\longrightarrow \dots 	\longrightarrow\bigoplus_{j\in \ZZ}  S(-j)^{\beta_{0,j}(\mathcal{M})}\longrightarrow \mathcal{M} \longrightarrow\ 0.$$
 Let $\reg(\mathcal{M})$ denotes the \textit{Castelnuovo-Mumford regularity} (or simply \textit{regularity}) of $\mathcal{M}$. Then $$\reg(\mathcal{M})=\max\{j-i:\beta_{i,j}(\mathcal{M})\neq 0\}.$$ Let $\pdim(\mathcal{M})$ denotes the \textit{projective dimension} of $\mathcal{M}$. Then $$\pdim(\mathcal{M})=\max\{i:\beta_{i,j}(\mathcal{M})\neq 0\}.$$ For values, bounds and some other interesting results related to these two invariants we refer the readers to  \cite{regnpro, regul, bakht, circulent}. If $\mathfrak{m}:=(x_1,\dots,x_q)$ be the unique maximal graded ideal of $S$, then the \textit{depth} of $\mathcal{M}$ is defined to be the common length of all maximal $M$-sequences in $\mathfrak{m}$.
Let $\mathcal{M}$ be a finitely generated $\mathbb{Z}^q$-graded $S$-module. Let $uK[Z]$ be a $K$-subspace of $\mathcal{M}$ which is generated by all elements of the form $uy$ where $u$ is a homogeneous element in $\mathcal{M},$ $y$ is a monomial in $K[Z]$ and $Z\subseteq\{x_1,\dots,x_q\}$. If $uK[Z]$ is a free $K[Z]$-module then it is called a Stanley space of dimension $|Z|$. A decomposition $\mathcal{D}$ of $K$-vector space $\mathcal{M}$ as a finite direct sum of Stanley spaces is called a Stanley decomposition of $\mathcal{M}$. Let $\mathcal{D} \, : \, \mathcal{M} = \oplus_{j=1}^m u_j K[Z_j],$ 
	the Stanley depth of $\mathcal{D}$ is  $\sdepth(\mathcal{D})=\min\{|Z_j|:j=1,2,\dots,m\}$. The number
	$$\sdepth(\mathcal{M}):=\max \{\sdepth(\mathcal{D})\,:\mathcal{D}\text{ is a Stanley decomposition of} \, \mathcal{M}\},$$ is called the \textit{Stanley depth} of $\mathcal{M}$. Herzog et al. \cite{herz} gave a method to compute Stanley depth for modules of the type $I/J$, where $J\subset I$ are monomial ideals of $S$.   Recently, Ichim et al. \cite{Ichim2} gave method for computing Stanley depth of any finitely generated $\mathbb{Z}^{q}$-graded $S$-module. However, it is still a challenging task to compute the Stanley even by using these methods. Therefore, it is still important to give values and bounds for Stanley depth of modules. For some interesting results related to Stanley depth we refer the readers to \cite{MC8,MI,ZIA,KS,PFY}.  In 1982 Stanley conjectured in \cite{20} that $\sdepth({\mathcal{M}})\geq \depth({\mathcal{M}})$, this conjecture has been been proved in some special cases; see for instance \cite{hersurvey}. However, in 2016 Duval et al. \cite{21} showed that this conjecture is false for modules of the type $S/I$, where $I$ is a monomila ideal.

Let $G:=(V(G),E(G))$ be a graph with vertex set $V(G)=\{x_1,\dots,x_q\}$ and edge set $E(G).$ Throughout this work, all graphs are finite and simple. The \textit{edge ideal} $I(G)$ associated to $G$ is a squarefree monomial ideal, that is,
$I(G)=(x_{i}x_{j} : \{x_i, x_j\}\in E(G)).$  
 A graph $G$ on vertex set $\{x_1,\dots,x_q\}$ is said to be a \textit{path} of length $q-1$ if $E(G)=\{\{x_i,x_{i+1}\}:i\in\{1,\dots,q-1\}\}$. We denote the path of length $q-1$ is denoted by $\mathbb{P}_q.$ A graph $G$ on vertex set $\{x_1,\dots,x_q\}$ is said to be a \textit{cycle} of length $q$ if $E(G)=E(\mathbb{P}_q)\cup \{\{x_{q},x_1\}\}$. We denote the cycle of length $q$ by $\mathbb{C}_q.$ The \textit{degree} of a vertex of a graph $G$ is the number of edges that are incident to that vertex. A graph $G$ is said to be \textit{$q$-regular} if every vertex of $G$ has degree $q$. A graph $\mathcal{T}$ is said to be a \textit{tree} if there exists a unique path between any two vertices of $\mathcal{T}$.
 A vertex of degree $1$ of a graph is called a \textit{pendant vertex} (or \textit{leaf}).  An \textit{internal vertex} is a vertex that is not a leaf. Let $q\geq 2$, a tree with one internal vertex and $q-1$ leaves incident on it is called a \textit{$q$-star}, we denote a $q$-star by $\mathbb{S}_q$. 
A  graph in which every pair of vertices is connected by an edge is called a \textit{complete graph}. We denote a complete graph on $q$ vertices by $\mathbb{K}_{q}$.
 A vertex $x_j$ is called a \textit{neighbor} of a vertex $x_i$ in a graph $G$ if $\{x_i,x_j\}\in E(G).$ The \textit{neighborhood} of a vertex $x_i$ in a graph $G$ denoted by $N_G(x_i)$  is defined to be the set of all neighbors of $x_i.$  A graph $\mathcal{H}$ is said to be a \textit{subgraph} of a graph $G$, if $V(\mathcal{H}) \subseteq V(G)$ and $E(\mathcal{H}) \subseteq E(G)$. If  $\mathcal{H}$ is a subgraph of $G,$ then $G$ is said to be a supergraph of  $\mathcal{H}.$

 Let $q\geq 2$ and  $\mathbb{S}$ be a subset of $\{1, \dots ,\lfloor\frac{q}{2}\rfloor\}.$ A \textit{circulant graph} $C_{q}(\mathbb{S})$ is a graph on the vertex set $\{x_{1}, \dots, x_{q}\}$ such that $\{x_{i}, x_{j} \} \in E(C_{q}(\mathbb{S}))$ if and only if $|i-j|$ or $q-|i-j|\in \mathbb{S}$. See Figure \ref{fig:c7(1,3)} for examples of circulant graphs. Since $\mathbb{C}_q=C_q(1)$ therefore circulant graphs are sometimes considered as generalized cycles. For convenience the graph $C_q(\{a_1,\dots,a_l\})$ is simply denoted by $C_q(a_1,\dots,a_l)$. A circulant graph $C_q(a_1,\dots,a_l)$ is $2l$-regular, except if $2a_l=q,$ in which case, it is $(2l-1)$-regular. As a consequence  $3$-regular circulant graphs have the form $C_{2n}(a,n)$ with $1\leq a\leq n.$   A 3-regular circulant graph is also called a cubic circulant graph.      
\noindent Several algebraic invariants and other  algebraic properties of the edge ideals of circulant graphs have already been studied; see for instance \cite{wellcover,betti, rinaldo,circulent,cohenmac}. 
 Circulant graphs have also applications in network theory \cite{net,mulnet}, in group theory \cite{isograph} and in the theory of designs and error-correcting codes \cite{error}. 
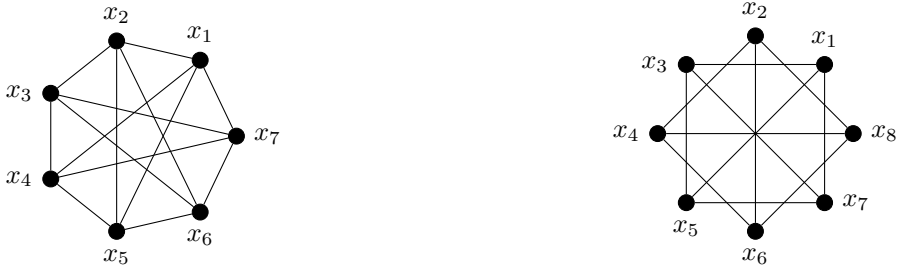
\begin{figure}[H]
	\centering
	\begin{subfigure}[b]{0.45\textwidth}
		\centering
	\[\begin{tikzpicture}[x=1.3cm, y=1.3cm]
		\vertex[fill] (c1) at (51:1) [label=above:$x_{1}$]{};
		\vertex[fill] (c2) at (103:1) [label=above:$x_{2}$]{};
		\vertex[fill] (c3) at (154:1) [label=left:$x_{3}$]{};
		\vertex[fill] (c4) at (206:1) [label=left:$x_{4}$]{};
		\vertex[fill] (c5) at (257:1) [label=below:$x_{5}$]{};
		\vertex[fill] (c6) at (309:1) [label=below:$x_{6}$]{};
		\vertex[fill] (c7) at (360:1) [label=right:$x_{7}$]{};
		\path 
		(c6) edge (c7)
		(c7) edge (c1)
		(c5) edge (c6)
		(c5) edge (c4)
		(c4) edge (c3)
		(c2) edge (c3)
		(c1) edge (c2)
		(c1) edge (c4)
		(c1) edge (c5)
		(c2) edge (c5)
		(c2) edge (c6)
		(c3) edge (c6)
		(c3) edge (c7)
		(c7) edge (c4)
		;
		\end{tikzpicture}\]
	\end{subfigure}
	\hfill
	\begin{subfigure}[b]{0.45\textwidth}
				\centering
	\[\begin{tikzpicture}[x=1.3cm, y=1.3cm]
		\vertex[fill] (c1) at (45:1) [label=above:$x_{1}$]{};
		\vertex[fill] (c2) at (90:1) [label=above:$x_{2}$]{};
		\vertex[fill] (c3) at (135:1) [label=left:$x_{3}$]{};
		\vertex[fill] (c4) at (180:1) [label=left:$x_{4}$]{};
		\vertex[fill] (c5) at (225:1) [label=below:$x_{5}$]{};
		\vertex[fill] (c6) at (270:1) [label=below:$x_{6}$]{};
		\vertex[fill] (c7) at (315:1) [label=right:$x_{7}$]{};
  \vertex[fill] (c8) at (360:1) [label=right:$x_{8}$]{};
		\path 
		(c1) edge (c3)
		(c1) edge (c7)
		(c5) edge (c1)
		(c2) edge (c4)
		(c2) edge (c8)
		(c2) edge (c6)
		(c3) edge (c5)
		(c3) edge (c7)
		(c4) edge (c8)
		(c4) edge (c6)
		(c5) edge (c7)
		(c6) edge (c8)
		;
		\end{tikzpicture}\]
	\end{subfigure}
	\hfill
	\caption{From left to right $C_{7}(1,3) $ and $ C_{8}(2,4) $.}\label{fig:c7(1,3)}
\end{figure}		

The work in this paper is inspired by a recent work of Uribe-Paczka et al. \cite{circulent}, where the authors study Castelnuovo-Mumford regularity of edge ideals of cubic circulant graphs. In this paper, we give values for depth and projective dimension, and lower bounds for Stanley depth of the quotient rings of the edge ideals of all cubic circulant graphs, see Theorem \ref{cy1}, Corollary \ref{cy2} and Theorem \ref{cy3}.   We also give a result in Lemma \ref{lemiso}, which is inspired by 
a result of Cimpoeas {\cite[Proposition 1.3] {cycledepth}}.  Lemma \ref{lemiso}  plays a vital role in the computation of depth and a lower bound for Stanley depth in our main findings. For proving over main results the precise values of the said invariants of the quotient rings of the edge ideals associated with certain supergraphs of ladder graph play a crucial role; see for instance Lemma \ref{The1}, Lemma \ref{diamond} and Lemma \ref{dotfamily}.  We gratefully acknowledge the use of CoCoA \cite{cocoa} and Macaulay 2 \cite{macaulay} for experiments.
\section{Preliminaries}
In this section, we present some results that are frequently used throughout the paper. 

\begin{Lemma}[{Depth Lemma}]\label{le01}
 If $0\rightarrow \mathcal{U} \rightarrow \mathcal{V} \rightarrow \mathcal{W} \rightarrow 0$ is a short exact sequence of modules over a
local ring $S$, or a Noetherian graded ring with $S_0$ local, then
\begin{enumerate}
\item $\depth (\mathcal{V}) \geq \min\{\depth(\mathcal{W}), \depth(\mathcal{U})\}$.
\item $\depth (\mathcal{U}) \geq \min\{\depth(\mathcal{V}), \depth(\mathcal{W}) + 1\}$.
\item $\depth (\mathcal{W}) \geq \min\{\depth(\mathcal{U})-1, \depth(\mathcal{V})\}$.
\end{enumerate}
\end{Lemma} \begin{Lemma}[{\cite[Corollary 1.3]{AR1}}]\label{Cor7}
Let $I\subset S$ be a monomial ideal and $u$ be a monomial such that $u\notin I.$ Then
$\depth (S/(I : u))\geq \depth(S/I).$
\end{Lemma} \begin{Lemma}[{\cite[Theorem 4.3]{regul}}]\label{exacther}
		Let $I$ be a monomial ideal and let $u$ be an arbitrary monomial in $S.$ Then 
		$$\depth(S/I)=\depth(S/(I:u))  \,\,\text{if}\,\, \depth(S/(I,u))\geq \depth(S/(I:u)).$$
	\end{Lemma}
\noindent Similar results hold for Stanley depth as follows.
\begin{Lemma}[{\cite{AR1}}]\label{le1}
Let $0\rightarrow \mathcal{U}\rightarrow \mathcal{V}\rightarrow \mathcal{W}\rightarrow 0$ be a short exact sequence of $\ZZ^{n}$-graded $S$-modules. Then
\begin{equation*}
\sdepth(\mathcal{V})\geq\min\{\sdepth(\mathcal{U}),\sdepth(\mathcal{W})\}.
\end{equation*}
\end{Lemma}

\begin{Lemma}[{\cite[Proposition 2.7]{MC}}]\label{Pro7}
Let $I\subset S$ be a monomial ideal and $u$ be a monomial such that $u\notin I.$ Then $\sdepth (S/(I : u))\geq \sdepth(S/I).$
\end{Lemma}
	
	\begin{Lemma}[{\cite[Lemma 1.7]{sdepth}}]\label{exacthersdepth}
		Let $I$ be a monomial ideal and let $u$ be a monomial in $S$ such that $u\notin I.$ Then 
		$$\sdepth(S/I)=\sdepth(S/(I:u))  \,\,\text{if}\,\, \sdepth(S/(I,u))\geq \sdepth(S/(I:u)).$$
	\end{Lemma}
 \noindent When we add new variables to the ring then depth and Stanley depth will likewise increase {\cite[Lemma 3.6]{herz}}. This fact is summarized in the following lemma.
		\begin{Lemma}\label{111}
			Let $J\subset S$ be a monomial ideal, and $\bar{S}=S \tensor_{K} K[x_{q+1}]$ be a polynomial ring in $q+1$ variables. Then  $\depth(\bar{S}/J)=\depth(S/J)+1$ and $   \sdepth(\bar{S}/J)=\sdepth(S/J)+1.$ 
		\end{Lemma}
  
	\begin{Lemma}[{\cite[Proposition 2.2.20]{book}}]\label{virtens}
		Let $1\leq r < q$ and $S= \mathcal{S}_1\tensor_K \mathcal{S}_2,$ where $\mathcal{S}_{1}=K[x_{1},\dots, x_{r}]$ and $\mathcal{S}_{2}=K[x_{r+1},\dots, x_{q}].$ If $I\subset \mathcal{S}_{1}$ and $J\subset \mathcal{S}_{2}$ are monomial ideals, then $S/(I+J)\cong \mathcal{S}_{1}/I \tensor_{K}\mathcal{S}_{2}/J.$ 
	\end{Lemma}
\noindent Combining Lemma \ref{virtens} with {\cite[Proposition 2.2.21]{book}} and {\cite[Theorem 3.1]{AR1}} for depth and Stanley depth, respectively, we get the following useful result.
	\begin{Lemma}\label{LEMMA1.5} Let $1\leq r < q$ and $S= \mathcal{S}_1\tensor_K \mathcal{S}_2,$ where $\mathcal{S}_{1}=K[x_{1},\dots, x_{r}]$ and $\mathcal{S}_{2}=K[x_{r+1},\dots, x_{q}].$  If $I\subset \mathcal{S}_{1}$ and $J\subset \mathcal{S}_{2}$ are monomial ideals, then
		$ \depth_{S}(\mathcal{S}_{1}/I\tensor_{K}\mathcal{S}_{2}/J)= \depth_{S}(S/(I+J))=\depth_{\mathcal{S}_{1}}(\mathcal{S}_{1}/I)+\depth_{\mathcal{S}_{2}}(\mathcal{S}_{2}/J)$ and 	$ \sdepth_{S}(\mathcal{S}_{1}/I \tensor_{K} \mathcal{S}_{2}/J)=\sdepth_{S}(S/(I+J)) \geq \sdepth_{\mathcal{S}_{1}}(\mathcal{S}_{1}/I)+\sdepth_{\mathcal{S}_{2}}(\mathcal{S}_{2}/J).$
	\end{Lemma}
 \noindent It is obvious and well known that $\depth(S)=\sdepth(S)=q.$ 
	  \begin{Lemma}[{\cite[Theorem 1.4]{miii}}]\label{stan1}
       Let $\mathcal{M}$ be a $\mathbb{Z}^q$-graded $S$-module. If $\sdepth(\mathcal{M})=0$ then  $\depth(\mathcal{M})=0.$  Conversely, if $\depth(\mathcal{M})=0$ and $\dim_K (\mathcal{M}_a)\leq 1$ for any $a\in\mathbb{Z}^q,$ then $\sdepth(\mathcal{M})=0.$
  \end{Lemma} 
	\begin{Lemma}[{\cite[Theorems 1.3.3]{depth}}](Auslander–Buchsbaum formula)\label{auss13}
		Let $ R $ be  a commutative Noetherian local ring and $\mathcal{M}$  be a non-zero finitely generated R-module of finite projective dimension. Then
		\begin{equation*}
			{\pdim}(\mathcal{M})+{\depth}(\mathcal{M})={\depth}(R).
		\end{equation*}
	\end{Lemma}
 \noindent The following lemma combines two result for depth and Stanley depth proved in \cite[Lemma 2.8]{pathsusan} and  \cite[Lemma 4]{stef}, respectively.
\begin{Lemma}\label{paath}
    Let $ q\geq2$. If $ I=I(\mathbb{P}_q)$, then $\depth(S/I)=\sdepth(S/I)=\lceil\frac{q}{3}\rceil.$
\end{Lemma} 	
\begin{Lemma}[{\cite[Proposition 1.3 and Theorem 1.9]{cycledepth}}]\label{cyccc}
    Let $q\geq 3$. If $I=I(\mathbb{C}_{q}),$ then\begin{itemize}
        \item[(a)] 	$\depth(S/I)=\lceil\frac{q-1}{3}\rceil.$ 
          \item[(b)] 	$\sdepth(S/I)=\lceil\frac{q-1}{3}\rceil,$ for $q\equiv 0,2(\mod 3)$ and $$\lceil\frac{q-1}{3}\rceil \leq \sdepth(S/I)\leq \lceil\frac{q}{3}\rceil, \,\,\text{for}\,\, q\equiv 1(\mod 3).$$
    \end{itemize} 
\end{Lemma} 
\begin{Lemma}[{\cite[Theorem 2.6 and Theorem 2.7]{AA}}]\label{leAli}
			Let $q\geq 2$. If $I=I(\mathbb{S}_{q})$, then $$\depth(S/I)=\sdepth(S/I))=1.$$
		\end{Lemma} \noindent We recall the following result  proved in  \cite[Corollary 10.3.7]{herhib} for depth and for Stanley depth in \cite[Theorem 1.1]{MC8}. \begin{Lemma}\label{com}
			Let $q\geq 2$. If $I=I(\mathbb{K}_{q})$, then $\depth(S/I)=\sdepth(S/I)=1.$
		\end{Lemma}

\section{Invariants of cyclic modules associated to some supergraphs of ladder graph}
For $n\geq 2,$ the graph $A_n$ as shown in Figure \ref{fig:1} is called  ladder graph  on $2n$ vertices. We introduce some supergraphs of ladder graph namely $B_n,C_n$ and $D_n$ that play a significant role in our main results. It will be convenient to label the vertices of the aforementioned graphs as shown in Figure \ref{fig:1} and Figure \ref{fig:2}. The vertex sets and edge sets of these graphs are:
 \begin{itemize}
     
     \item $V(A_n)=\underset{i=1}{\overset{n}{\cup}}\{x_i,y_i\},$  $E(A_{n})=\underset{i=1}{\overset{n-1}{\cup}}\{\{x_i,y_i\}, \{x_i,x_{i+1}\},\{y_i,y_{i+1}\} \}\cup \{x_n,y_n\},$
     \item $V(B_n)=V(A_n)\cup \{y_{n+1}\},$ $E(B_n)=E(A_n)\cup \{y_n,y_{n+1}\},$ 
     \item $V(C_{n})=V(A_{n})\cup\big\{y_{n+1},y_{n+2}\big\},$ $E(C_{n})=E(A_{n})\cup\big\{\{y_{n},y_{n+1}\},\{y_{1},y_{n+2}\}\big\},$ 
     \item $V(D_{n})=V(A_{n})\cup\big\{x_{n+1},y_{n+1}\big\},$    $E(D_{n})=E(A_{n})\cup\big\{\{y_{n},y_{n+1}\},\{x_{1},x_{n+1}\}\big\}.$
 \end{itemize}

\begin{figure}[H]
	\centering
	\begin{subfigure}[b]{0.45\textwidth}
		\centering
	\[\begin{tikzpicture}[x=0.8cm, y=0.7cm]
 	\vertex[fill] (6) at (0,0) [label=below:${x_{1}}$] {};
 	\vertex[fill] (7) at (1,0) [label=below:${x_{2}}$] {};
 	\vertex[fill] (8) at (2,0) [label=below:${x_{3}}$] {};
 	\vertex[fill] (9) at (3,0) [label=below:${x_{4}}$] {};
 	\vertex[fill] (10) at (4,0) [label=below:${x_{n-2}}$] {};
  \vertex[fill] (11) at (5,0) [label=below:${x_{n-1}}$] {};
   \vertex[fill] (12) at (6,0) [label=below:${x_{n}}$] {};
 	\vertex[fill] (6a) at (0,1) [label=above:${y_{1}}$] {};
 	\vertex[fill] (7a) at (1,1) [label=above:${y_{2}}$] {};
 	\vertex[fill] (8a) at (2,1) [label=above:${y_{3}}$] {};
 	\vertex[fill] (9a) at (3,1) [label=above:${y_{4}}$] {};
 	\vertex[fill] (10a) at (4,1) [label=above:${y_{n-2}}$] {};
 	\vertex[fill] (11a) at (5,1) [label=above:${y_{n-1}}$] {};
  \vertex[fill] (12a) at (6,1) [label=above:${y_{n}}$] {};
  \draw (3,1) node {} -- (4,1) [dashed] node {};
	\draw (3,0) node {} -- (4,0) [dashed] node {};
 	\path 
 	(6) edge (7)
 	(7) edge (8)
 	(8) edge (9)
 	(6a) edge (7a)
 	(7a) edge (8a)
 	(8a) edge (9a)
 	(6) edge (6a)
  (10) edge (11)
  (11) edge (12)
  (11) edge (11a)
  (11a) edge (12a)
  (12) edge (12a)
 	(7) edge (7a)
 	(8) edge (8a)
 	(9) edge (9a)
 	(10) edge (10a)
 	(11a) edge (10a)
 	;
 	\end{tikzpicture}\]
	\end{subfigure}
	\hfill
	\begin{subfigure}[b]{0.45\textwidth}
		\centering
	\[\begin{tikzpicture}[x=0.8cm, y=0.7cm]
 	\vertex[fill] (6) at (0,0) [label=below:${x_{1}}$] {};
 	\vertex[fill] (7) at (1,0) [label=below:${x_{2}}$] {};
 	\vertex[fill] (8) at (2,0) [label=below:${x_{3}}$] {};
 	\vertex[fill] (9) at (3,0) [label=below:${x_{4}}$] {};
 	\vertex[fill] (10) at (4,0) [label=below:${x_{n-2}}$] {};
  \vertex[fill] (11) at (5,0) [label=below:${x_{n-1}}$] {};
   \vertex[fill] (12) at (6,0) [label=below:${x_{n}}$] {};
 	\vertex[fill] (6a) at (0,1) [label=above:${y_{1}}$] {};
 	\vertex[fill] (7a) at (1,1) [label=above:${y_{2}}$] {};
 	\vertex[fill] (8a) at (2,1) [label=above:${y_{3}}$] {};
 	\vertex[fill] (9a) at (3,1) [label=above:${y_{4}}$] {};
 	\vertex[fill] (10a) at (4,1) [label=above:${y_{n-2}}$] {};
 	\vertex[fill] (11a) at (5,1) [label=above:${y_{n-1}}$] {};
  \vertex[fill] (12a) at (6,1) [label=above:${y_{n}}$] {};
  \vertex[fill] (13a) at (7,1) [label=above:${y_{n+1}}$] {};
  \draw (3,1) node {} -- (4,1) [dashed] node {};
	\draw (3,0) node {} -- (4,0) [dashed] node {};
 	\path 
 	(6) edge (7)
 	(7) edge (8)
 	(8) edge (9)
 	(6a) edge (7a)
 	(7a) edge (8a)
 	(8a) edge (9a)
 	(6) edge (6a)
  (10) edge (11)
  (11) edge (12)
  (12a) edge (13a)
  (11) edge (11a)
  (11a) edge (12a)
  (12) edge (12a)
 	(7) edge (7a)
 	(8) edge (8a)
 	(9) edge (9a)
 	(10) edge (10a)
 	(11a) edge (10a)
 	;
 	\end{tikzpicture}\]
	\end{subfigure}
	\hfill
	\caption{From left to right $A_n $ and $ B_n $.}
	\label{fig:1}
\end{figure}
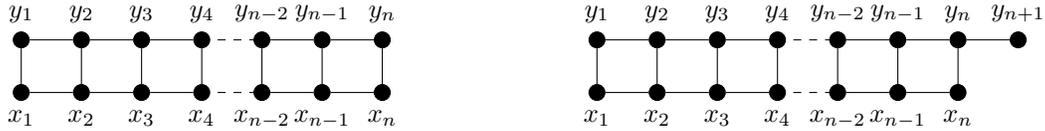
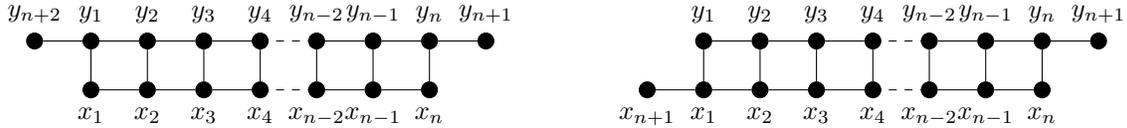
\begin{figure}[H]
	\centering
	\begin{subfigure}[b]{0.45\textwidth}
		\centering
	\[\begin{tikzpicture}[x=0.75cm, y=0.65cm]
 	\vertex[fill] (6) at (0,0) [label=below:${x_{1}}$] {};
 	\vertex[fill] (7) at (1,0) [label=below:${x_{2}}$] {};
 	\vertex[fill] (8) at (2,0) [label=below:${x_{3}}$] {};
 	\vertex[fill] (9) at (3,0) [label=below:${x_{4}}$] {};
 	\vertex[fill] (10) at (4,0) [label=below:${x_{n-2}}$] {};
  \vertex[fill] (11) at (5,0) [label=below:${x_{n-1}}$] {};
   \vertex[fill] (12) at (6,0) [label=below:${x_{n}}$] {};
       \vertex[fill] (66a) at (-1,1) [label=above:${y_{n+2}}$] {};
 	\vertex[fill] (6a) at (0,1) [label=above:${y_{1}}$] {};
 	\vertex[fill] (7a) at (1,1) [label=above:${y_{2}}$] {};
 	\vertex[fill] (8a) at (2,1) [label=above:${y_{3}}$] {};
 	\vertex[fill] (9a) at (3,1) [label=above:${y_{4}}$] {};
 	\vertex[fill] (10a) at (4,1) [label=above:${y_{n-2}}$] {};
 	\vertex[fill] (11a) at (5,1) [label=above:${y_{n-1}}$] {};
  \vertex[fill] (12a) at (6,1) [label=above:${y_{n}}$] {};
  \vertex[fill] (13a) at (7,1) [label=above:${y_{n+1}}$] {};
  \draw (3,1) node {} -- (4,1) [dashed] node {};
	\draw (3,0) node {} -- (4,0) [dashed] node {};
 	\path 
 	(6) edge (7)
 	(7) edge (8)
 	(8) edge (9)
 	(6a) edge (7a)
 	(7a) edge (8a)
 	(8a) edge (9a)
 	(6) edge (6a)
  (10) edge (11)
  (11) edge (12)
  (12a) edge (13a)
  (11) edge (11a)
  (11a) edge (12a)
  (12) edge (12a)
 	(7) edge (7a)
 	(8) edge (8a)
 	(9) edge (9a)
 	(10) edge (10a)
 	(11a) edge (10a)
  (7a) edge (66a)
 	;
 	\end{tikzpicture}\]
	\end{subfigure}
	\hfill
	\begin{subfigure}[b]{0.45\textwidth}
		\centering
	\[\begin{tikzpicture}[x=0.75cm, y=0.65cm]
 \vertex[fill] (66) at (-1,0) [label=below:${x_{n+1}}$] {};
 	\vertex[fill] (6) at (0,0) [label=below:${x_{1}}$] {};
 	\vertex[fill] (7) at (1,0) [label=below:${x_{2}}$] {};
 	\vertex[fill] (8) at (2,0) [label=below:${x_{3}}$] {};
 	\vertex[fill] (9) at (3,0) [label=below:${x_{4}}$] {};
 	\vertex[fill] (10) at (4,0) [label=below:${x_{n-2}}$] {};
  \vertex[fill] (11) at (5,0) [label=below:${x_{n-1}}$] {};
   \vertex[fill] (12) at (6,0) [label=below:${x_{n}}$] {};
 	\vertex[fill] (6a) at (0,1) [label=above:${y_{1}}$] {};
 	\vertex[fill] (7a) at (1,1) [label=above:${y_{2}}$] {};
 	\vertex[fill] (8a) at (2,1) [label=above:${y_{3}}$] {};
 	\vertex[fill] (9a) at (3,1) [label=above:${y_{4}}$] {};
 	\vertex[fill] (10a) at (4,1) [label=above:${y_{n-2}}$] {};
 	\vertex[fill] (11a) at (5,1) [label=above:${y_{n-1}}$] {};
  \vertex[fill] (12a) at (6,1) [label=above:${y_{n}}$] {};
  \vertex[fill] (13a) at (7,1) [label=above:${y_{n+1}}$] {};
  \draw (3,1) node {} -- (4,1) [dashed] node {};
	\draw (3,0) node {} -- (4,0) [dashed] node {};
 	\path 
 	(6) edge (7)
  (6) edge (66)
 	(7) edge (8)
 	(8) edge (9)
 	(6a) edge (7a)
 	(7a) edge (8a)
 	(8a) edge (9a)
 	(6) edge (6a)
  (10) edge (11)
  (11) edge (12)
  (12a) edge (13a)
  (11) edge (11a)
  (11a) edge (12a)
  (12) edge (12a)
 	(7) edge (7a)
 	(8) edge (8a)
 	(9) edge (9a)
 	(10) edge (10a)
 	(11a) edge (10a)
 	;
 	\end{tikzpicture}\]
	\end{subfigure}
	\hfill
	\caption{From left to right $C_n $ and $ D_n $.}
	\label{fig:2}
\end{figure} \noindent In this section, we compute  the values of depth, Stanley depth and projective dimension of the cyclic modules  $K[V(B_n)]/I(B_n),$  $K[V(C_n)]/I(C_n)$ and $K[V(D_n)]/I(D_n).$ These values play a vital role in our main results in last section. For cyclic module $K[V(A_n)]/I(A_n),$ we give exact value of Stanley depth when $n\equiv 1(\mod 2)$  and find sharp bounds when $n\equiv 0(\mod 2).$  First, we prove the following lemma that will help in computing depth and lower bound for Stanley depth, throughout the paper. As mentioned earlier this lemma is inspired by {\cite[Proposition 1.3]{cycledepth}}.
\begin{Lemma}\label{lemiso}
    Let $G$ be a  connected graph with $V(G)=\{x_1,\dots, x_n\}.$ If $N_G(x_i)=\{x_{i_1},\dots, x_{i_l}\}$, then $$(I(G):x_i)/I(G)\cong \bigoplus_{t=1}^{l} S_t/J_t[x_{i_t}],$$ where
     $S_1=K[V(G)\backslash N_G(x_{i_1})]$ for $t\geq 2$, $S_t=K[V(G)\backslash \big(N_G(x_{i_t})\cup \{x_{i_1},x_{i_2},\dots,x_{i_{t-1}}\}\big)]$ and for $t\geq 1$,  $J_t=(S_t \cap I(G)).$
     
\end{Lemma} \begin{proof}
    If $u\in (I(G):x_{i})$ is a monomial such that $u\notin I(G),$ then it follows that $u$ is divisible by at least one variable from $N_G(x_i)= \{x_{i_1},\dots, x_{i_l}\}.$ Indeed, if $u$ is not divisible by any of the variables from the set of $N_G(x_{i})$ then $u\in I(G),$ a contradiction. Without loss of generality we may assume that $x_{i_1}|u$ then $u=x^{\alpha_1}_{i_1}v_{1}$ with $\alpha_1\geq 1.$ Since $u\notin I(G)$, it follows that $v_{1}\in S_1=K[V(G) \backslash N_G(x_{i_1})] $ and   $v_1 \notin J_1=(S_1\cap I(G)).$  Thus $u \in x_{i_1} (S_1/J_1)[ x_{i_1}].$ Now, if $x_{i_2}|u$ and $x_{i_1}\nmid u$, then  $u=x^{\alpha_2}_{i_2}v_{2}$ with $\alpha_2\geq 1.$ It follows that $v_{2}\in S_2=K[V(G) \backslash (N_G(x_{i_2})\cup \{x_{i_1}\})]$ and   $v_2 \notin J_2=(S_2\cap I(G)).$  Thus $u \in x_{i_2} (S_2/J_2)[ x_{i_2}].$ In a similar manner,  for  $3\leq t \leq  l,$ if  $x_{i_t}|u$ and $x_{i_1}\nmid u, x_{i_2}\nmid u$,  \dots,  $x_{i_{t-1}}\nmid u$ then  $u=x^{\alpha_t}_{i_t}v_{t}$ with $\alpha_t\geq 1.$ Since $u\notin I(G)$, it follows that $v_{t}\in S_t=K[V(G)\backslash \big(N_G(x_{i_t})\cup  \{x_{i_1},x_{i_2},\dots,x_{i_{t-1}}\}\big)]$ and   $v_t \notin J_t=( S_t\cap I(G)).$  Thus $u \in x_{i_t} (S_t/J_t)[ x_{i_t}]$
and we have the following $S$-module isomorphism
$$(I(G):x_i)/I(G)\cong \bigoplus_{t=1}^{l}x_{i_t} (S_t/J_t)[x_{i_t}].$$ It is easy to see that $ x_{i_t}$ is regular on  $S_t/J_t[x_{i_t}], $ therefore we have $ x_{i_t} (S_t/J_t)[x_{i_t}] \cong (S_t/J_t)[x_{{i_t}}].$ This completes the proof. 
\end{proof}
\begin{Remark}\label{rem1}
   \em{If $n\leq 1$, then we define the quotient rings associated to $A_n$, $B_n$, $C_n$ and $D_n$ appearing in the proofs of this section as follows: 
   \begin{itemize}

       \item $K[V(B_0)]/I(B_0) = K[x]$   and   $\depth(K[x])=\sdepth(K[x])=1.$
       
       \item $ K[V(A_1)]/I(A_1) = K[V(\mathbb{P}_2)]/I(\mathbb{P}_2) $ and  by using Lemma \ref{paath},  $\depth(K[V(\mathbb{P}_2)]/I(\mathbb{P}_2)) =\sdepth(K[V(\mathbb{P}_2)]/I(\mathbb{P}_2))= 1.$
       \item  $K[V(B_1)]/I(B_1)= K[V(\mathbb{P}_3)]/I(\mathbb{P}_3),$ by Lemma \ref{paath} we have  $\depth( K[V(\mathbb{P}_3)]/I(\mathbb{P}_3))=\sdepth( K[V(\mathbb{P}_3)]/I(\mathbb{P}_3))=1.$ 
       
       \item $K[V(C_1)]/I(C_1)= K[V(\mathbb{S}_4)]/I(\mathbb{S}_4)$ and by Lemma \ref{leAli}, we get $ \depth(K[V(\mathbb{S}_4)]/I(\mathbb{S}_4))=\sdepth(K[V(\mathbb{S}_4)]/I(\mathbb{S}_4))=1.$
       \item $K[V(D_1)]/I(D_1)= K[V(\mathbb{P}_4)]/I(\mathbb{P}_4)$ and by using Lemma \ref{paath},  $\depth(K[V(\mathbb{P}_4)]/I(\mathbb{P}_4))=\sdepth(K[V(\mathbb{P}_4)]/I(\mathbb{P}_4))=2.$ 
   \end{itemize}}\end{Remark} 


\noindent Let $\mathbb{G}(I)$ denotes the minimal set of monomial generators of monomial ideal $I.$ For a monomial ideal $I$, $\supp{(I)}:=\{x_{i}:x_{i}|u\, \,  \text{for some} \,\,  u \in \mathbb{G}(I)\}.$	
	\begin{Remark}
{\em Let $I$ be a squarefree monomial ideal of $S$ minimally generated by monomials of degree at most $2.$ We associate a graph $G_{I}$ to the ideal $I$ with $V(G_I)=\supp(I)$  and $E(G_I)=\{\{x_{i},x_{j}\}:x_{i}x_{j}\in \mathbb{G}(I)\}.$ Let $x_{t},x_l \in S$ be a variable of the polynomial ring $S$ such that $x_{t},x_l \notin I.$ Then $(I:x_{t}),$  $(I,x_{t}),$  $((I,x_{t}),x_l)$ and  $((I,x_{t}):x_l)$ are the monomial ideals of $S$ such that $G_{(I:x_{t})},$ $G_{(I,x_{t})},$ $G_{((I,x_{t}),x_l)}$ and $G_{((I,x_{t}):x_l)}$ are subgraphs of $G_{I}.$ 
By using the labeling of Figure \ref{fig:2},  see for instance; Figure  \ref{fig:11} and \ref{fig:21} as examples of $G_{(I(C_{7}): x_{7})},G_{(I(C_{7}), x_{7})},G_{((I(C_{7}), x_{7}),y_8)}$  and $G_{((I(C_{7}), x_{7}):y_8)}.$
From Figures \ref{fig:11} and \ref{fig:21}, after suitable renumbering of variables we have the following isomorphisms:
	\begin{equation*}
	   	K[V(C_{7})]/(I(C_{7}): x_{7}) \cong K[V(C_{5})]/I(C_5) \tensor_K K[x_7,y_8],
	\end{equation*}
	\begin{equation*}
	   	K[V(C_{7})]/(I(C_{7}), x_{7}) \cong K[V(C_{6})]/(I(C_6),y_7y_8),
	\end{equation*}\begin{equation*}
	   	K[V(C_{7})]/((I(C_{7}), x_{7}),y_8) \cong K[V(C_{6})]/I(C_6),
	\end{equation*}  \begin{equation*}
	   	K[V(C_{7})]/((I(C_{7}), x_{7}):y_8) \cong K[V(B_{6})]/I(B_6) \tensor_K K[y_8].
	\end{equation*}

\begin{figure}[H]
	\centering
	\begin{subfigure}[b]{0.45\textwidth}
			\centering
	\[\begin{tikzpicture}[x=0.7cm, y=0.6cm]
 	\vertex[fill] (6) at (0,0) [label=below:${x_{1}}$] {};
 	\vertex[fill] (7) at (1,0) [label=below:${x_{2}}$] {};
 	\vertex[fill] (8) at (2,0) [label=below:${x_{3}}$] {};
 	\vertex[fill] (9) at (3,0) [label=below:${x_{4}}$] {};
 	\vertex[fill] (10) at (4,0) [label=below:${x_{5}}$] {};
  \vertex[fill] (11) at (5,0) [label=below:${x_{6}}$] {};
       \vertex[fill] (66a) at (-1,1) [label=above:${y_{9}}$] {};
 	\vertex[fill] (6a) at (0,1) [label=above:${y_{1}}$] {};
 	\vertex[fill] (7a) at (1,1) [label=above:${y_{2}}$] {};
 	\vertex[fill] (8a) at (2,1) [label=above:${y_{3}}$] {};
 	\vertex[fill] (9a) at (3,1) [label=above:${y_{4}}$] {};
 	\vertex[fill] (10a) at (4,1) [label=above:${y_{5}}$] {};
 	\vertex[fill] (11a) at (5,1) [label=above:${y_{6}}$] {};
  \vertex[fill] (12a) at (6,1) [label=above:${y_{7}}$] {};
 	\path 
(9) edge (10)
(9a) edge (10a)
 	(6) edge (7)
 	(7) edge (8)
 	(8) edge (9)
 	(6a) edge (7a)
 	(7a) edge (8a)
 	(8a) edge (9a)
 	(6) edge (6a)
 	(7) edge (7a)
 	(8) edge (8a)
 	(9) edge (9a)
 	(10) edge (10a)
 	(11a) edge (10a)
  (7a) edge (66a)
 	;
 	\end{tikzpicture}\]
	\end{subfigure}
	\hfill
	\begin{subfigure}[b]{0.45\textwidth}
		\centering
	\[\begin{tikzpicture}[x=0.7cm, y=0.6cm]
 	\vertex[fill] (6) at (0,0) [label=below:${x_{1}}$] {};
 	\vertex[fill] (7) at (1,0) [label=below:${x_{2}}$] {};
 	\vertex[fill] (8) at (2,0) [label=below:${x_{3}}$] {};
 	\vertex[fill] (9) at (3,0) [label=below:${x_{4}}$] {};
 	\vertex[fill] (10) at (4,0) [label=below:${x_{5}}$] {};
  \vertex[fill] (11) at (5,0) [label=below:${x_{6}}$] {};
   \vertex[fill] (12) at (6,0) [label=below:${x_{7}}$] {};
       \vertex[fill] (66a) at (-1,1) [label=above:${y_{9}}$] {};
 	\vertex[fill] (6a) at (0,1) [label=above:${y_{1}}$] {};
 	\vertex[fill] (7a) at (1,1) [label=above:${y_{2}}$] {};
 	\vertex[fill] (8a) at (2,1) [label=above:${y_{3}}$] {};
 	\vertex[fill] (9a) at (3,1) [label=above:${y_{4}}$] {};
 	\vertex[fill] (10a) at (4,1) [label=above:${y_{5}}$] {};
 	\vertex[fill] (11a) at (5,1) [label=above:${y_{6}}$] {};
  \vertex[fill] (12a) at (6,1) [label=above:${y_{7}}$] {};
  \vertex[fill] (13a) at (7,1) [label=above:${y_{8}}$] {};
  
 	\path 
(9) edge (10)
(9a) edge (10a)
 	(6) edge (7)
 	(7) edge (8)
 	(8) edge (9)
 	(6a) edge (7a)
 	(7a) edge (8a)
 	(8a) edge (9a)
 	(6) edge (6a)
  (10) edge (11)
 (12a) edge (13a)
  (11) edge (11a)
  (11a) edge (12a)
 	(7) edge (7a)
 	(8) edge (8a)
 	(9) edge (9a)
 	(10) edge (10a)
 	(11a) edge (10a)
  (7a) edge (66a)
 	;
 	\end{tikzpicture}\]
	\end{subfigure}
	\hfill
	\caption{From left to right $G_{(I(C_{7}): x_{7})}$ and $G_{(I(C_{7}), x_{7})}$.}
	\label{fig:11}
\end{figure}
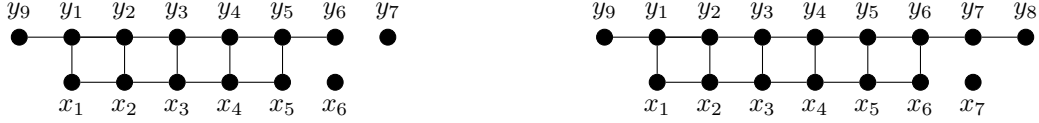		

\begin{figure}[H]
	\centering
	\begin{subfigure}[b]{0.45\textwidth}
		\centering
	\[\begin{tikzpicture}[x=0.7cm, y=0.6cm]
 	\vertex[fill] (6) at (0,0) [label=below:${x_{1}}$] {};
 	\vertex[fill] (7) at (1,0) [label=below:${x_{2}}$] {};
 	\vertex[fill] (8) at (2,0) [label=below:${x_{3}}$] {};
 	\vertex[fill] (9) at (3,0) [label=below:${x_{4}}$] {};
 	\vertex[fill] (10) at (4,0) [label=below:${x_{5}}$] {};
  \vertex[fill] (11) at (5,0) [label=below:${x_{6}}$] {};
   \vertex[fill] (12) at (6,0) [label=below:${x_{7}}$] {};
       \vertex[fill] (66a) at (-1,1) [label=above:${y_{9}}$] {};
 	\vertex[fill] (6a) at (0,1) [label=above:${y_{1}}$] {};
 	\vertex[fill] (7a) at (1,1) [label=above:${y_{2}}$] {};
 	\vertex[fill] (8a) at (2,1) [label=above:${y_{3}}$] {};
 	\vertex[fill] (9a) at (3,1) [label=above:${y_{4}}$] {};
 	\vertex[fill] (10a) at (4,1) [label=above:${y_{5}}$] {};
 	\vertex[fill] (11a) at (5,1) [label=above:${y_{6}}$] {};
  \vertex[fill] (12a) at (6,1) [label=above:${y_{7}}$] {};
  \vertex[fill] (13a) at (7,1) [label=above:${y_{8}}$] {};
  
 	\path 
(9) edge (10)
(9a) edge (10a)
 	(6) edge (7)
 	(7) edge (8)
 	(8) edge (9)
 	(6a) edge (7a)
 	(7a) edge (8a)
 	(8a) edge (9a)
 	(6) edge (6a)
  (10) edge (11)
  (11) edge (11a)
  (11a) edge (12a)
 	(7) edge (7a)
 	(8) edge (8a)
 	(9) edge (9a)
 	(10) edge (10a)
 	(11a) edge (10a)
  (7a) edge (66a)
 	;
 	\end{tikzpicture}\]
	\end{subfigure}
	\hfill
	\begin{subfigure}[b]{0.45\textwidth}
			\centering
	\[\begin{tikzpicture}[x=0.7cm, y=0.6cm]
 	\vertex[fill] (6) at (0,0) [label=below:${x_{1}}$] {};
 	\vertex[fill] (7) at (1,0) [label=below:${x_{2}}$] {};
 	\vertex[fill] (8) at (2,0) [label=below:${x_{3}}$] {};
 	\vertex[fill] (9) at (3,0) [label=below:${x_{4}}$] {};
 	\vertex[fill] (10) at (4,0) [label=below:${x_{5}}$] {};
  \vertex[fill] (11) at (5,0) [label=below:${x_{6}}$] {};
   \vertex[fill] (12) at (6,0) [label=below:${x_{7}}$] {};
       \vertex[fill] (66a) at (-1,1) [label=above:${y_{9}}$] {};
 	\vertex[fill] (6a) at (0,1) [label=above:${y_{1}}$] {};
 	\vertex[fill] (7a) at (1,1) [label=above:${y_{2}}$] {};
 	\vertex[fill] (8a) at (2,1) [label=above:${y_{3}}$] {};
 	\vertex[fill] (9a) at (3,1) [label=above:${y_{4}}$] {};
 	\vertex[fill] (10a) at (4,1) [label=above:${y_{5}}$] {};
 	\vertex[fill] (11a) at (5,1) [label=above:${y_{6}}$] {};
  \vertex[fill] (12a) at (6,1) [label=above:${y_{7}}$] {};
  
 	\path 
(9) edge (10)
(9a) edge (10a)
 	(6) edge (7)
 	(7) edge (8)
 	(8) edge (9)
 	(6a) edge (7a)
 	(7a) edge (8a)
 	(8a) edge (9a)
 	(6) edge (6a)
  (10) edge (11)
  (11) edge (11a)
 	(7) edge (7a)
 	(8) edge (8a)
 	(9) edge (9a)
 	(10) edge (10a)
 	(11a) edge (10a)
  (7a) edge (66a)
 	;
 	\end{tikzpicture}\]
	\end{subfigure}
	\hfill
	\caption{From left to right $G_{((I(C_{7}), x_{7}),y_8)} $ and $ G_{((I(C_{7}), x_{7}):y_8)}$.}
	\label{fig:21}
\end{figure}
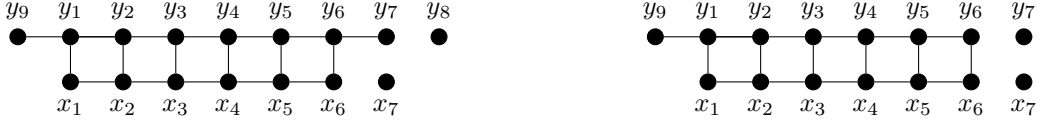 } 
\end{Remark}

\begin{Lemma}\label{The1}
Let $n\geq 2$ and  $S:=K[V(B_n)].$ Then
$\depth(S/I(B_n))=\sdepth(S/I(B_n))=\lceil\frac{n+1}{2}\rceil.$ 
\end{Lemma}
\begin{proof}
 First we prove the result for depth. Consider the following short exact sequence
\begin{equation*}\label{es1}
0\longrightarrow S/(I(B_n):y_{n})\xrightarrow{\cdot y_{n}} S/I(B_n)\longrightarrow S/(I(B_n),y_{n})\longrightarrow 0.
\end{equation*}
After a suitable numbering of variables we have the following isomorphisms 
\begin{equation}\label{ff1}
    S/(I(B_n):y_{n})\cong K[V(B_{n-2})]/(I(B_{n-2}))\tensor_K K[y_n],
\end{equation}
\begin{equation}\label{ff2}
    S/(I(B_n),y_{n})\cong  K[V(B_{n-1})]/(I(B_{n-1}))\tensor_K K[y_{n+1}].
\end{equation}
If $n=2,$ then by Lemma \ref{111} and Remark \ref{rem1}, $\depth(S/(I(B_2):y_{2}))=\depth(K[V(B_{0})]/I(B_{0}))+1=\depth(K[x_1])+1=2$ and $$\depth(S/(I(B_2),y_{2}))=\depth(K[V(B_{1})]/I(B_{1}))+1=\depth(K[V(\mathbb{P}_3)]/I(\mathbb{P}_3))+1=2.$$ By Lemma \ref{exacther}, $\depth(K[V(B_2)]/I(B_2))=2.$ Similarly, the required result follows when $n=3.$  Let $n\geq 4.$ By induction on $n$ and applying Lemma \ref{111} on Equations (\ref{ff1}) and  (\ref{ff2}), it follows that $$\depth(S/(I(B_n):y_{n}))=\depth(K[V(B_{n-2})]/I(B_{n-2}))+1= \lceil\frac{n-2+1}{2}\rceil+1=\lceil\frac{n+1}{2}\rceil,$$
 $$\depth(S/(I(B_n),y_{n}))=\depth(K[V(B_{n-1})]/I(B_{n-1}))+1 =\lceil\frac{n-1+1}{2}\rceil+1=\lceil\frac{n}{2}\rceil +1.$$ 
Clearly,  $\depth(S/(I(B_n),y_{n}))\geq \depth(S/(I(B_n):y_{n}))$, therefore by Lemma \ref{exacther}, we get the required result. 	For Stanley depth, the proof is similar to depth by using Lemma \ref{exacthersdepth} in place of Lemma \ref{exacther}.
\end{proof}
\begin{Corollary}
    Let $n\geq 2$ and  $S:=K[V(B_n)].$ Then
$\pdim(S/I(B_n))=2n-\lceil\frac{n+1}{2}\rceil+1.$ 
\end{Corollary} \begin{proof}
    The proof follows by Lemma \ref{auss13}  and Lemma \ref{The1}.
\end{proof}
\noindent The projective dimension of $K[V(A_n)]/I(A_n)$ has been computed in \cite{ladder}.
\begin{Theorem}[{\cite[Theorem 5.5]{ladder}}]\label{ladd}
     If $n\geq 2$, then $\pdim(K[V(A_n)]/I(A_n))=\lfloor\frac{3n}{2}\rfloor.$
\end{Theorem}\noindent Consequently, one can compute its depth by using Auslander–Buchsbaum formula.  \begin{Corollary}
   If $n\geq 2,$ then $\depth(K[V(A_n)]/I(A_n))=2n-\lfloor\frac{3n}{2}\rfloor=\lceil\frac{n}{2}\rceil.$
\end{Corollary} \begin{proof}
    The required result follows by using  Lemma \ref{auss13}  and Theorem \ref{ladd}.
\end{proof}  \noindent Here we give an alternative proof for depth by using Lemma \ref{lemiso} we include this proof because proof for Stanley depth is analogous. In Remark \ref{rem2} we explain a situations that arises in special cases in upcoming proofs.

\begin{Remark}\label{rem2}
\em{ Let $i\in\mathbb{Z^+},$ if $k<i$ then we consider $\cup^{k}_{i}\{x_{i}y_{i},x_{i}x_{i+1},y_{i}y_{i+1}\}=\emptyset$. Also we take $x_ay_b=0$, whenever $a$ or $b$ is not positive.}  
\end{Remark}

\begin{Theorem}\label{The2}
    If $n\geq 2$ and $S=K[V(A_n)],$  then $\sdepth(S/I(A_{n}))\geq\depth(S/I(A_{n}))=\lceil\frac{n}{2}\rceil.$
\end{Theorem}
\begin{proof}
 We first prove the result for depth.
 If $n=2,$  then by Lemma \ref{cyccc},  $\depth(K[V(A_2)]/I(A_{2}))=1.$ Let $n\geq 3$. Consider the following short exact sequence
\begin{equation}\label{es2q}
0\longrightarrow (I(A_{n}):y_{n})/I(A_{n})\xrightarrow{\cdot y_{n}} S/I(A_{n})\longrightarrow S/(I(A_{n}):y_{n})\longrightarrow 0.
\end{equation} We have the following $S$-module isomorphism
\begin{equation}\label{sss1}
    S/(I(A_n):y_{n}) \cong K[V(B_{n-2})]/I(B_{n-2})\tensor_K K[y_{n}].
\end{equation} Here $N_{A_n}(y_n)=\{y_{n-1},x_n\},$ $S_1=K[V(A_n)\backslash N_{A_n}(y_{n-1})],$ $S_2=K[V(A_n)\backslash (N_{A_n}(x_{n})\cup \{y_{n-1}\})],$ $J_1=(S_1\cap I(A_n)),$ $J_2=(S_2\cap I(A_n)),$ then by using Lemma \ref{lemiso}, we have 
 \begin{equation}\label{GHJ}
     \begin{split}
        (I(A_n):y_{n})/I(A_n)&\cong S_1/J_1[y_{n-1}] \oplus S_2/J_2[x_{n}]\\&\cong  \frac{K[x_{1},\dots,x_{n-2},x_n,y_{1},\dots,y_{n-3}]}{\big(\cup^{n-4}_{i=1}\{x_{i}y_{i},x_{i}x_{i+1},y_{i}y_{i+1}\}\cup\{x_{n-3}y_{n-3},x_{n-3}x_{n-2} \}\big)}[y_{n-1}]
\\ & \quad\oplus \frac{K[x_{1},\dots,x_{n-2},y_{1},\dots,y_{n-2}]}{\big(\cup^{n-3}_{i=1}\{x_{i}y_{i},x_{i}x_{i+1},y_{i}y_{i+1}\}\cup\{x_{n-2}y_{n-2}\}\big)}[x_{n}]\\\\&\cong K[V(B_{n-3})]/I(B_{n-3})\tensor_K K[x_n,y_{n-1}]\oplus K[V(A_{n-2})]/I(A_{n-2}) \tensor_K K[x_n].
     \end{split}
 \end{equation} If $n=3,$ then we have \begin{equation}\label{gg1}
     K[V(A_3)]/(I(A_3):y_{3})\cong \frac{K[x_1,x_2,y_1]}{(y_1x_1,x_1x_2)}[y_3] \cong K[V(B_1)]/I(B_1) \tensor_K K[y_{3}],
 \end{equation}
 \begin{equation}\label{gg2} \begin{split}
     (I(A_3):y_{3})/I(A_3) &\cong \frac{K[x_1,x_3]}{(0)} [y_2] \oplus \frac{K[x_1,y_1]}{(x_1y_1)}[x_3] \\&\cong K[x_1,x_3,y_2]\oplus K[V(A_1)]/I(A_1)\tensor_K K[x_3].
 \end{split}\end{equation}  By applying Lemma \ref{111}  on Equations (\ref{gg1}) and (\ref{gg2}) and using Remark \ref{rem1} we have $$\depth(K[V(A_3)]/(I(A_3):y_{3}))=\depth(K[V(B_1)]/I(B_1))+1=\depth(K[V(\mathbb{P}_3)]/I(\mathbb{P}_3))+1=2$$  and 
 \begin{equation*}
     \begin{split}
         \depth((I(A_3):y_{3})/I(A_3))& =\min\{\depth( K[x_1,x_3,y_2]),\depth( K[V(A_1)]/I(A_1))+1)\}\\&=\min\{3,\depth(K[V(\mathbb{P}_2)]/I(\mathbb{P}_2))+1)\}\\&=\min\{3,2\}= 2.
     \end{split}
 \end{equation*}
  Thus by applying Depth Lemma on Equation (\ref{es2q}), we get $\depth(K[V(A_3)]/I(A_3))=2.$ If $n=4,$ as stated in Remark \ref{rem2}, we have $\cup^{n-4}_{i=1}\{x_{i}y_{i},x_{i}x_{i+1},y_{i}y_{i+1}\}=\emptyset$ and using the similar strategy and induction on $n,$ one can get the required result. Let $n\geq 5.$   
 By applying Lemma \ref{111} and Lemma \ref{The1} on Equation (\ref{sss1}), we get $$\depth(S/(I(A_n):y_{n}))=\depth(K[V(B_{n-2})]/I(B_{n-2}))+1= \lceil\frac{n-2+1}{2}\rceil+1=\lceil\frac{n+1}{2}\rceil.$$ By  using Equation (\ref{GHJ}) and applying Lemma \ref{111}, Lemma \ref{The1} and induction on $n$, we have
\begin{equation*}
    \begin{split} \depth((I(A_{n}):y_{n})/I(A_{n}))&=\min\Big\{\depth(K[V(B_{n-3})]/I(B_{n-3}))+\depth(K[x_n,y_{n-1}]), \\& \quad \quad\quad\quad\depth(K[V(A_{n-2})]/I(A_{n-2}))+\depth(K[x_n])\Big\}\\& =\min\Big\{\lceil\frac{n-3+1}{2}\rceil+2,\lceil\frac{n-2}{2}\rceil+1\Big\}= \lceil\frac{n}{2}\rceil.
    \end{split}
\end{equation*} We get the required result by using Depth Lemma on  Equation (\ref{es2q}). Now we prove the result for Stanley depth.  If $n=2,$  then by Lemma \ref{cyccc},  we have $\sdepth(K[V(A_2)]/I(A_{2}))\geq 1.$ If $n=3,$ then  by using Lemma \ref{111} and  Remark \ref{rem1} on Equations (\ref{gg1}) and  (\ref{gg2}) $$\sdepth(K[V(A_3)]/(I(A_3):y_{3}))=\sdepth(K[V(B_1)]/I(B_1))+1=\sdepth(K[V(\mathbb{P}_3)]/I(\mathbb{P}_3))+1=2$$  and 
 \begin{equation*}
     \begin{split}
         \sdepth((I(A_3):y_{3})/I(A_3))& \geq \min\{\sdepth( K[x_1,x_2,y_2]),\sdepth( K[V(A_1)]/I(A_1))+1)\}\\&\geq\min\{3,\sdepth(K[V(\mathbb{P}_2)]/I(\mathbb{P}_2))+1)\}\\&\geq\min\{3,1+1\}= 2.
     \end{split}
 \end{equation*}
  Thus by applying Lemma \ref{le1} on Equation (\ref{es2q}),  we get $\sdepth(K[V(A_3)]/I(A_3))\geq 2.$
For $n\geq 4$, we get the required lower bound for Stanley depth by using the similar arguments just by using Lemma \ref{le1} in place of Depth Lemma on the exact sequence (\ref{es2q}).
\end{proof}
	
\begin{Corollary}\label{Cor1}
    Let $n\geq 2$ and $S=K[V(A_n)].$ If \,$n\equiv 0\,(\mod 2),$ then $\sdepth(S/I(A_n))\in \{\lceil\frac{n}{2}\rceil, \lceil\frac{n+1}{2}\rceil \},$   otherwise we have $\sdepth(S/I(A_n))=\lceil\frac{n}{2}\rceil.$\end{Corollary}
\begin{proof} If $n=2,$ then one can easily see that the result follows by Theorem \ref{The2} and  Lemma \ref{cyccc}. If $n\geq 3$, then by Theorem \ref{The2}, we only need to show that $\sdepth(S/I(A_n))\leq \lceil\frac{n+1}{2}\rceil$. For $y_n \notin I(A_n),$ and by using Lemma \ref{Pro7}, we have
$\sdepth(S/I(A_n))\leq \sdepth(S/(I(A_n):y_{n})).$  By applying   Lemma \ref{111} and Lemma \ref{The1} on Equation (\ref{sss1}), we get $\sdepth(S/(I(A_n):y_{n})= \lceil\frac{n-2+1}{2}\rceil+1=\lceil\frac{n+1}{2}\rceil$  and the required result follows.
\end{proof} 	

		\begin{Lemma}\label{diamond}
		    Let $n\geq 2$ and $S=K[V(C_n)]$. Then 
			
			\begin{equation*}
				\depth(S/I(C_n))=\sdepth(S/I(C_n))=\left\{\begin{matrix}
				\lceil\frac{n}{2} \rceil+1, & \,\,\,\,\,\text{if}\,\,\, \,n\equiv 0,3\, (\mod\, 4);\\ \\
					\lceil\frac{n+1}{2} \rceil, & \text{if}\,\,\, \, n\equiv 1 \, (\mod\, 4);
					\\	\\\lceil\frac{n+1}{2} \rceil+1, & \text{if}\,\,\,\,\, \, n\equiv 2\, (\mod\, 4).
				\end{matrix}\right.
			\end{equation*} 
		\end{Lemma}
		
		\begin{proof} First, we prove the result for depth. If $n=2,$ then consider the following short exact sequence \begin{equation}\label{dd1}
0\longrightarrow S/(I(C_2):y_{2})\xrightarrow{\cdot y_{2}} S/I(C_2)\longrightarrow S/(I(C_2),y_{2})\longrightarrow 0.
\end{equation} Here $S/(I(C_2):y_2)\cong K[x_1,y_2,y_4]$ and  $S/(I(C_2),y_2)\cong K[y_3] \tensor_K K[V(\mathbb{P}_4)]/I(\mathbb{P}_4).$  By using  Lemma \ref{111} and Lemma \ref{paath}, we get $\depth(S/(I(C_2):y_2))=\depth( K[x_1,y_2,y_4])=3$ and $\depth(S/(I(C_2),y_2))=\depth(  K[y_3])+\depth(K[V(\mathbb{P}_4)]/I(\mathbb{P}_4))=3.$  By applying Lemma \ref{exacther} on Equation (\ref{dd1}), we get  $\depth(S/(I(C_2))=3=\lceil\frac{2+1}{2} \rceil+1.$ Let $n\geq 3$ and consider the following short exact sequences \begin{equation*}
0\longrightarrow S/(I(C_n):x_{n})\xrightarrow{\cdot x_{n}} S/I(C_n)\longrightarrow S/(I(C_n),x_{n})\longrightarrow 0,
\end{equation*} \begin{equation*}
0\longrightarrow S/\big((I(C_n),x_{n}):y_{n+1}\big)\xrightarrow{\cdot y_{n+1}} S/(I(C_n),x_{n})\longrightarrow S/\big((I(C_n),x_{n}),y_{n+1}\big)\longrightarrow 0, 
\end{equation*} and  by Depth Lemma  \begin{equation}\label{es322}
\depth(S/I(C_n)\geq \min \big\{\depth\big(S/(I(C_n):x_{n})\big), \depth\big(S/(I(C_n),x_{n})\big)\big\},
\end{equation}\begin{equation}\label{es3222} 
\depth(S/(I(C_n),x_{n}))\geq \min \big\{\depth\big(S/((I(C_n),x_{n}):y_{n+1})\big), \depth \big(S/((I(C_n),x_{n}),y_{n+1})\big)\big\}.
\end{equation} After a suitable numbering of variables, we have the following $K$-algebra isomorphisms:
 \begin{equation}\label{s112}
    S/(I(C_{n}):x_{n}) \cong K[V(C_{n-2})]/I(C_{n-2}) \tensor_K K[x_{n},y_{n+1}].
\end{equation} 
\begin{equation}\label{s113}
    S/\big((I(C_n),x_{n}),y_{n+1}\big)\cong K[V(C_{n-1})]/I(C_{n-1}),
\end{equation} \begin{equation}\label{s114}
    S/\big((I(C_n),x_{n}):y_{n+1}\big)\cong K[V(B_{n-1})]/I(B_{n-1}) \tensor_K K[y_{n+1}].
\end{equation} If $n=3,$ we have by Equation (\ref{s112}), \begin{equation}\label{fhj}
    S/(I(C_{3}):x_{3}) \cong K[V(C_{1})]/I(C_{1})  \tensor_K K[x_{3},y_{4}]. \end{equation} By Lemma \ref{111} and Remark \ref{rem1}, we have  $\depth(S/(I(C_{3}):x_{3}))=\depth(K[V(C_{1})]/I(C_{1}))+2=\depth(K[V(\mathbb{S}_4)]/I(\mathbb{S}_4))+2=3.$  By using Equations (\ref{s113}) and   (\ref{s114}),    $ S/\big((I(C_3),x_{3}),y_{4}\big)\cong K[V(C_{2})]/I(C_{2})$ and $S/\big((I(C_3),x_{3}):y_{4}\big)\cong K[V(B_{2})]/I(B_{2}) \tensor_K K[y_{4}].$ By using induction on $n,$ we have $\depth\big(S/((I(C_3),x_{3}),y_{4})\big)=\depth(K[V(C_{2})]/I(C_{2}))=3$ and by Lemma \ref{111} and Lemma  \ref{The1},  we get $\depth\big(S/((I(C_3),x_{3}):y_{4})\big)=\depth(K[V(B_{2})]/I(B_{2}))+1=3.$ By   Equation (\ref{es3222}), $\depth(S/(I(C_3),x_{3}))\geq 3$ and by Equation (\ref{es322}), $\depth(S/(I(C_3))\geq 3.$ For the upper bound, we use  Lemma \ref{Cor7} and Equation (\ref{fhj}), that is $\depth(S/(I(C_3))\leq  \depth\big(S/(I(C_3):x_{3})\big)=3. $ Thus we get $\depth(S/(I(C_3))=\lceil\frac{3}{2}\rceil+1=3.$ If  $n=4,$ by using similar strategy as we did when $n=3,$ we get the required lower bound that is $\depth(S/I(C_{4}))\geq \lceil \frac{4}{2}\rceil+1=3.$  For the upper bound,  since $x_{3}y_{4}\notin I(C_4),$  we have  $$  S/(I(C_{4}):x_{3}y_{4}) \cong  K[V(\mathbb{S}_4)]/I(\mathbb{S}_4) \tensor_K K[x_{3},y_{4}],$$  and by Lemma \ref{111} and Lemma \ref{leAli}, $\depth(S/(I(C_{4}):x_{3}y_{4}))=\depth( K[V(\mathbb{S}_4)]/I(\mathbb{S}_4))+2=3.$ Therefore, by  Lemma \ref{Cor7}, $\depth(S/I(C_{4}))\leq \depth(S/(I(C_{4}):x_{3}y_{4}))=3.$  Let $n\geq  5.$  We consider the following cases:
  \begin{description} \item[Case 1] Let $n\equiv 1(\mod 4).$ We  consider the short exact sequence
\begin{equation}\label{es31hh}
    0\longrightarrow S/(I(C_n):y_{n+1})\xrightarrow{\cdot y_{n+1}} S/I(C_n)\longrightarrow S/(I(C_n),y_{n+1})\longrightarrow 0,\end{equation} 
Here $S/(I(C_n),y_{n+1}) \cong K[V(B_n)]/I(B_n)$. By Lemma \ref{The1}, $\depth(S/(I(C_n),y_{n+1}))= \lceil\frac{n+1}{2}\rceil.$   Consider another short exact sequence
\begin{equation}\label{es32hh}\begin{split}
0\longrightarrow S/\big((I(C_n):y_{n+1}):x_{n}\big)\xrightarrow{\cdot x_{n}} &S/(I(C_n):y_{n+1})\longrightarrow  \\ &S/\big((I(C_n):y_{n+1}),x_{n}\big)\longrightarrow 0.\end{split}
\end{equation} Here $S/\big((I(C_n):y_{n+1}),x_{n}\big)\cong K[V(B_{n-1})]/I(B_{n-1}) \tensor_K K[y_{n+1}].$ By Lemma \ref{111} and Lemma \ref{The1}, it follows that $\depth\big(S/((I(C_n):y_{n+1}),x_{n})\big)= \lceil\frac{n-1+1}{2}\rceil+1=\lceil\frac{n}{2}\rceil+1.$ Also we have  $S/\big((I(C_n):y_{n+1}):x_{n}\big)\cong K[V(C_{n-2})]/I(C_{n-2})\tensor_K K[y_{n+1},x_{n}].$ Since $n-2\equiv 3(\mod 4),$ by induction on $n$ and  Lemma \ref{111}, we get $\depth(S/((I(C_n):y_{n+1}):x_{n}))= \lceil\frac{n-2}{2}\rceil+1+2=\lceil\frac{n}{2}\rceil+2.$ By applying Depth Lemma on Equations (\ref{es31hh}) and  (\ref{es32hh}) \begin{equation}\label{vv2}
\depth(S/(I(C_n))\geq \min\big \{\depth(S/(I(C_n):y_{n+1})), \depth(S/(I(C_n),y_{n+1}))\big\},
\end{equation} \begin{equation}\label{vv1}\begin{split}
\depth(S/(I(C_n):y_{n+1}))\geq &\min\big \{\depth\big(S/((I(C_n):y_{n+1}):x_{n})\big), \\&\quad\quad\quad\depth\big(S/((I(C_n):y_{n+1}),x_{n})\big)\big\},
\end{split}
\end{equation}  By Equation (\ref{vv1}), $\depth(S/(I(C_n):y_{n+1}))\geq \lceil\frac{n}{2}\rceil+1.$ Since,  $\depth(S/(I(C_n):y_{n+1})) > \depth(S/(I(C_n),y_{n+1})),$ by Equation (\ref{vv2})
 we get $  \depth(S/(I(C_n))\geq \lceil\frac{n+1}{2}\rceil.$ For the other inequality, we have  $x_{n-2}y_n \notin I(C_n),$ and the following $K$-algebra isomorphism: \begin{equation*}
    S/(I(C_{n}):x_{n-2}y_{n}) \cong K[V(C_{n-4})]/I(C_{n-4}) \tensor_K K[x_{n-2},y_{n}].
\end{equation*} Since $n-4\equiv 1(\mod 4),$ by Remark \ref{rem1}, Lemma \ref{Cor7}, Lemma \ref{111}  and induction on $n,$ we get $\depth(S/I(C_{n}))\leq \depth(S/(I(C_{n}):x_{n-2}y_{n}))=\lceil\frac{n-4+1}{2}\rceil+2=\lceil\frac{n+1}{2}\rceil,$ as required.
				\item[Case 2] Let $n\equiv 2(\mod 4).$ Consider the following short exact sequences: \begin{equation*}
0\longrightarrow S/(I(C_n):x_{n-1})\xrightarrow{\cdot x_{n-1}} S/I(C_n)\longrightarrow S/(I(C_n),x_{n-1})\longrightarrow 0,
\end{equation*} \begin{equation*}\begin{split}
0\longrightarrow S/\big((I(C_n),x_{n-1}):y_{n-1}\big)&\xrightarrow{\cdot y_{n-1}} S/(I(C_n),x_{n-1})\\& \longrightarrow S/\big((I(C_n),x_{n-1}),y_{n-1}\big)\longrightarrow 0,\end{split}
\end{equation*} \begin{equation*} \begin{split}
0\longrightarrow S/\big(((I(C_n),x_{n-1}):y_{n-1}):x_{n-2}\big) &\xrightarrow{\cdot x_{n-2}} S/\big((I(C_n),x_{n-1}):y_{n-1}\big)\\& \longrightarrow S/\big(((I(C_n),x_{n-1}):y_{n-1}),x_{n-2}\big)\longrightarrow 0.
\end{split}
\end{equation*} We have the following $K$-algebra isomorphisms: $$S/(I(C_{n}):x_{n-1}) \cong K[V(C_{n-3})]/I(C_{n-3}) \tensor_K K[x_{n-1}] \tensor_K K[V(\mathbb{P}_2)]/I(\mathbb{P}_2),$$ $$S/\big((I(C_n),x_{n-1}),y_{n-1}\big)\cong K[V(B_{n-2})]/I(B_{n-2}) \tensor_K K[V(\mathbb{P}_3)]/I(\mathbb{P}_3),$$ $$S/\big(((I(C_n),x_{n-1}):y_{n-1}):x_{n-2}\big) \cong K[V(C_{n-4})]/I(C_{n-4}) \tensor_K K[x_{n-2},y_{n-1},x_n,y_{n+1}],$$
$$ S/\big(((I(C_n),x_{n-1}):y_{n-1}),x_{n-2}\big) \cong K[V(B_{n-3})]/I(B_{n-3}) \tensor_K K[y_{n-1},x_{n},y_{n+1}].$$ 
 Since  $n-3\equiv 3(\mod 4),$ by using induction on $n,$ Lemma \ref{LEMMA1.5} and Lemma \ref{paath}, we have $\depth(S/I(C_{n}):x_{n-1})=\depth(K[V(C_{n-3})]/I(C_{n-3}))  + \depth(K[V(\mathbb{P}_2)]/I(\mathbb{P}_2))+1=\lceil\frac{n-3}{2}\rceil+3=\lceil\frac{n+1}{2}\rceil+1.$   By using Lemma \ref{paath}, Lemma \ref{The1} and Lemma \ref{LEMMA1.5}, we have $\depth({S/((I(C_n),x_{n-1}),y_{n-1})})=\depth(K[V(B_{n-2})]/I(B_{n-2}) )+\depth(K[V(\mathbb{P}_3)]/I(\mathbb{P}_3))=\lceil\frac{n-2+1}{2}\rceil+2=\lceil\frac{n+1}{2}\rceil+1.$  Since $n-4\equiv 2(\mod 4)$ and $n-3\equiv 3(\mod 4).$  By induction on $n$ and Lemma \ref{Cor7}, we have \begin{equation*}
     \begin{split}
         \depth\big(S/\big(((I(C_n),x_{n-1}):y_{n-1}):x_{n-2}\big)\big)&=\depth( K[V(C_{n-4})]/I(C_{n-4}))+4\\&=\lceil\frac{n-4+1}{2}\rceil+1+4=\lceil\frac{n+3}{2}\rceil+2,
     \end{split}
 \end{equation*}
and by Lemma \ref{The1},
\begin{equation*}
    \begin{split}
        \depth\big(S/\big(((I(C_n),x_{n-1}):y_{n-1}),x_{n-2}\big)\big)&=\depth(K[V(B_{n-3})]/I(B_{n-3}) )+3\\&=\lceil\frac{n-3+1}{2}\rceil+3=\lceil\frac{n+1}{2}\rceil+1.
    \end{split}
\end{equation*}
By Depth Lemma on short exact sequences \begin{equation}\label{d44}
\depth(S/I(C_n)\geq \min \big\{\depth(S/(I(C_n):x_{n-1})), \depth(S/(I(C_n),x_{n-1}))\big\}.
\end{equation}\begin{equation}\label{D22}\begin{split}
\depth(S/(I(C_n),x_{n-1}))&\geq \min \Big\{\depth\big(S/((I(C_n),x_{n-1}):y_{n-1})\big),\\& \quad \quad\quad\quad\depth\big(S/((I(C_n),x_{n-1}),y_{n-1})\big)\Big\}.\end{split} \end{equation} \begin{equation}\label{d11} \begin{split}
\depth\big(S/((I(C_n),x_{n-1}):y_{n-1})\big)& \geq \min \Big\{\depth\big(S/(((I(C_n),x_{n-1}):y_{n-1}):x_{n-2})\big), \\&\quad\quad \quad \quad\depth\big(S/(((I(C_n),x_{n-1}):y_{n-1}),x_{n-2})\big)\Big\}.\end{split}\end{equation} 
Clearly $\lceil\frac{n+1}{2}\rceil+1 < \lceil\frac{n+3}{2}\rceil+2,$ by Equation (\ref{d11}) we have $\depth(S/((I(C_n),x_{n-1}):y_{n-1}))\geq \lceil\frac{n+1}{2}\rceil+1 .$ By Equation (\ref{D22}), we have $\depth(S/(I(C_n),x_{n-1}))\geq  \lceil\frac{n+1}{2}\rceil+1$ and by Equation (\ref{d44}), we get $\depth(S/I(C_n)\geq \lceil\frac{n+1}{2}\rceil+1.$ For the other inequality, we have  $x_{n} \notin I(C_n).$  Since $n-2\equiv 0(\mod 4),$ by using induction on $n,$  Lemma \ref{Cor7} and Lemma \ref{111} on Equation (\ref{s112}), we get $\depth(S/I(C_{n}))\leq \depth(S/(I(C_{n}):x_{n}))=\lceil\frac{n-2}{2}\rceil+1+2=\lceil\frac{n}{2}\rceil+2=\lceil\frac{n+1}{2}\rceil+1.$ 
				\item[Case 3] If $n\equiv 3(\mod 4).$  Since $n-2\equiv 1(\mod 4)$ and  $n-1\equiv 2(\mod 4),$ by induction on $n$ and Lemma \ref{111} on Equation (\ref{s112}), we get \begin{equation}\label{s111}\begin{split}\depth(S/(I(C_n):x_{n}))&=\depth( K[V(C_{n-2})]/I(C_{n-2}))+2\\&= \lceil\frac{n-2+1}{2}\rceil+2=\lceil\frac{n+1}{2}\rceil+1=\lceil\frac{n}{2}\rceil+1. 
				\end{split}\end{equation}   
By Equation   (\ref{s113}), we get  $\depth\big(S/((I(C_n),x_{n}),y_{n+1})\big)=\depth(K[V(C_{n-1})]/I(C_{n-1}))= \lceil\frac{n-1+1}{2}\rceil+1=\lceil\frac{n}{2}\rceil+1$ and by Lemma \ref{111} and Lemma \ref{The1} on  Equation (\ref{s114}), we get $\depth\big(S/((I(C_n),x_{n}):y_{n+1})\big)=\depth(K[V(B_{n-1})]/I(B_{n-1}))+1= \lceil\frac{n}{2}\rceil+1.$ Since $\depth\big(S/((I(C_n),x_{n}),y_{n+1})\big)=\depth\big(S/((I(C_n),x_{n}):y_{n+1})\big),$ therefore by  Equation (\ref{es3222}), $\depth(S/(I(C_n),x_{n}))\geq  \lceil\frac{n}{2}\rceil+1.$ Also $\depth(S/(I(C_n),x_{n}))= \depth(S/(I(C_n):x_{n})),$ by Equation (\ref{es322}) we get  $\depth(S/(I(C_n))\geq \lceil\frac{n}{2}\rceil+1.$ For the upper bound, we use  Lemma \ref{Cor7} and Equation (\ref{s111}), that is $\depth(S/(I(C_n))\leq  \depth(S/(I(C_n):x_{n}))=\lceil\frac{n}{2}\rceil+1, $ the required result.
				\item[Case 4] Let $n\equiv 0(\mod 4).$   In this case $n-2\equiv 2(\mod 4),$  by using induction on $n$ and Lemma \ref{111}  on Equation (\ref{s112}), we get 
    \begin{equation*}
    \depth(S/(I(C_n):x_{n}))=\depth( K[V(C_{n-2})]/I(C_{n-2}))+2=\lceil\frac{n-2+1}{2}\rceil+3=\lceil\frac{n+5}{2}\rceil.
\end{equation*}  
    As $n-1\equiv 3(\mod 4),$  by using induction on $n$ and Lemma \ref{111} on Equation  (\ref{s113}),    $$\depth\big(S/((I(C_n),x_{n}),y_{n+1})\big)= \depth(K[V(C_{n-1})]/I(C_{n-1}))= \lceil\frac{n-1}{2}\rceil+1=\lceil\frac{n}{2}\rceil+1.$$ By applying Lemma \ref{111} and Lemma \ref{The1}  on Equation  (\ref{s114}),  we have 
    \begin{equation*}
        \begin{split}
\depth\big(S/((I(C_n),x_{n}):y_{n+1})\big)&=\depth( K[V(B_{n-1})]/I(B_{n-1}))+1\\&= \lceil\frac{n-1+1}{2}\rceil+1=\lceil\frac{n}{2}\rceil+1.
        \end{split}
    \end{equation*} By  using Equation (\ref{es3222}), we get $\depth(S/(I(C_n),x_{n}))\geq \depth(S/(I(C_n),x_{n})=\lceil\frac{n}{2}\rceil+1.$ Here we have  $\depth(S/(I(C_n),x_{n}))\geq \depth(S/(I(C_n):x_{n})),$ thus by Equation (\ref{es322}), we get   $\depth(S/(I(C_n))\geq \lceil\frac{n}{2}\rceil+1.$  For the other inequality,   $x_{n-1}y_n \notin I(C_n)$ and consider \begin{equation}\label{bb1}
    S/(I(C_{n}):x_{n-1}y_{n}) \cong K[V(C_{n-3})]/I(C_{n-3}) \tensor_K K[x_{n-1},y_{n}].
\end{equation} Since $n-3\equiv 1(\mod 4),$ by using induction on $n,$  Lemma \ref{Cor7} and  Lemma \ref{111}  on Equation (\ref{bb1}),   $\depth(S/I(C_{n}))\leq \depth(S/(I(C_{n}):x_{n-1}y_n))=\depth( K[V(C_{n-3})]/I(C_{n-3}))+2=\lceil\frac{n-3+1}{2}\rceil+2=\lceil\frac{n}{2}\rceil+1.$ 
    \end{description} This completes the proof for depth. Proof for Stanley depth is similar as depth just by replacing Depth Lemma and Lemma \ref{Cor7} by Lemma \ref{le1} and 
 Lemma \ref{Pro7}. Also by using Lemma \ref{exacthersdepth} in place of Lemma \ref{exacther}.
\end{proof} \begin{Corollary}
     Let $n\geq 2$ and $S=K[V(C_n)]$. Then 
			
			\begin{equation*}
				\pdim(S/I(C_n))=\left\{\begin{matrix}
				2n-\lceil\frac{n}{2} \rceil+1, & \quad \text{if}\, \, n\equiv 0,3\, (\mod\, 4);\\ \\
					2n-\lceil\frac{n+1}{2} \rceil+2, &\quad \text{if}\, \, n\equiv 1 \, (\mod\, 4);
					\\	\\2n-\lceil\frac{n+1}{2} \rceil+1, &\quad \text{if}\, \, n\equiv 2\, (\mod\, 4).
				\end{matrix}\right.
			\end{equation*} 
\end{Corollary} \begin{proof}
    By using  Lemma \ref{auss13} and Lemma \ref{diamond}, the  result follows.
\end{proof}
	\begin{Lemma}\label{dotfamily}
		    Let $n\geq 2$ and $S=K[V(D_n)]$. Then 
			
			\begin{equation*}
				\depth(S/I(D_n))=\sdepth(S/I(D_n))=\left\{\begin{matrix}
				\lceil\frac{n+1}{2} \rceil+1, & \quad \text{if}\, \, n\equiv 0,1\, (\mod\, 4);\\ \\
					\lceil\frac{n+1}{2} \rceil, &\quad \text{if}\, \, n\equiv 2,3 \, (\mod\, 4).
				\end{matrix}\right.
			\end{equation*}
		\end{Lemma}
		
		\begin{proof}  We consider the following cases:
  \begin{description}

\item[Case 1] Let $n\equiv 2(\mod 4).$ Consider the short exact sequence
	\begin{equation}\label{es3133}
	0\longrightarrow S/(I(D_{n}):x_{n+1})\xrightarrow{\cdot x_{n+1}} S/I(D_{n})\longrightarrow S/(I(D_{n}),x_{n+1})\longrightarrow 0.
	\end{equation} We have
 \begin{equation}\label{f11}
     S/(I(D_{n}):x_{n+1})\cong K[V(C_{n-1})]/I(C_{n-1})\tensor_K K[x_{n+1}],
 \end{equation}
\begin{equation}\label{f22}
     S/(I(D_{n}),x_{n+1})\cong K[V(B_{n})]/I(B_{n}),
\end{equation}
Since $n-1\equiv 1(\mod 4).$ By using  Lemma \ref{111}, Lemma \ref{diamond} and Remark \ref{rem1}
\begin{equation*}
     \depth{( S/(I(D_{n}):x_{n+1}))} =\depth{(K[V(C_{n-1})]/I(C_{n-1}))}+1=\lceil\frac{n+2}{2}\rceil,
\end{equation*} and by Lemma \ref{The1} and Lemma \ref{LEMMA1.5}
\begin{equation}\label{f33}
    \depth{( S/(I(D_{n}),x_{n+1}))}=\depth{(K[V(B_{n})]/I(B_{n}))}=\lceil\frac{n+1}{2}\rceil.
\end{equation} Here $\lceil\frac{n+1}{2}\rceil=\lceil\frac{n+2}{2}\rceil,$ therefore  by Lemma \ref{exacther}, we get the required result. The proof for Stanley depth is similar by using Lemma \ref{exacthersdepth} in place of Lemma \ref{exacther}.

\item[Case 2] Let $n\equiv 3(\mod 4).$ 
If we consider Equation (\ref{es3133}), then by Depth Lemma \begin{equation*} \depth(S/I(D_{n}))\geq \min\{\depth(S/(I(D_{n}):x_{n+1})), \depth(S/(I(D_{n}),x_{n+1})\}. \end{equation*}
In this case $n-1\equiv 2(\mod 4),$ thus by using Equation (\ref{f11}) and applying  Lemma \ref{111} and Lemma \ref{diamond}, we get
\begin{equation*}
     \depth{( S/(I(D_{n}):x_{n+1}))} =\depth{(K[V(C_{n-1})]/I(C_{n-1}))}+1=\lceil\frac{n}{2}\rceil+2=\lceil\frac{n+2}{2}\rceil+1.
\end{equation*} Thus by using Equation (\ref{f33}) and Depth Lemma we get $\depth(S/I(D_{n}))\geq \lceil\frac{n+1}{2}\rceil.$ For the other inequality, since  $y_{n} \notin I(D_n),$ after suitable numbering of the variables, we have the following $K$- algebra isomorphism: \begin{equation}\label{E1}
    S/(I(D_{n}):y_{n}) \cong K[V(C_{n-2})]/I(C_{n-2}) \tensor_K K[y_{n}].
\end{equation} Since $n-2\equiv 1(\mod 4),$ by applying  Lemma \ref{Cor7}, Lemma \ref{111},  Lemma \ref{diamond} and Remark \ref{rem1} on Equation (\ref{E1}), $\depth(S/I(D_{n}))\leq \depth(S/(I(D_{n}):y_{n}))=\depth{(K[V(C_{n-2})]/I(C_{n-2}) )}+\depth{(K[y_{n}])} =\lceil\frac{n-2+1}{2}\rceil+1=\lceil\frac{n+1}{2}\rceil.$ For Stanley depth, by using  Lemma \ref{Pro7} instead of Lemma \ref{Cor7} and a similar strategy for depth, the required result is obtained. \item[Case 3] Let $n\equiv 0(\mod 4).$ We  consider the short exact sequence
\begin{equation*}\label{es31}
0\longrightarrow S/(I(D_n):y_{n})\xrightarrow{\cdot y_{n}} S/I(D_n)\longrightarrow S/(I(D_n),y_{n})\longrightarrow 0,
\end{equation*} 
After renumbering the variables, we have \begin{equation}\label{E2}
    S/(I(D_n),y_{n}) \cong K[V(C_{n-1})]/I(C_{n-1}) \tensor_K K[y_{n+1}].
\end{equation} 
 Since $n-1\equiv 3(\mod 4)$ and $n-2\equiv 2(\mod 4).$ By using Lemma  \ref{111}, Lemma \ref{diamond} on Equations (\ref{E1}) and (\ref{E2}), we get \begin{equation*} 
     \depth(S/(I(D_{n}):y_{n}))=\depth{(K[V(C_{n-2})]/I(C_{n-2}) )}+\depth{(K[y_{n}])}=\lceil\frac{n+1}{2}\rceil+1,
\end{equation*} and 
\begin{equation*}
     \depth(S/(I(D_{n}),y_{n}))=\depth{(K[V(C_{n-1})]/I(C_{n-1}) )}+\depth{(K[y_{n+1}])} =\lceil\frac{n+1}{2}\rceil+1.
\end{equation*} We have  $ \depth(S/(I(D_{n}),y_{n}))=\depth(S/(I(D_{n}):y_{n})),$ thus  by Lemma \ref{exacther}, we get $\depth(S/I(D_n))=\lceil\frac{n+1}{2}\rceil+1.$ The proof for Stanley depth is similar by using Lemma \ref{exacthersdepth} in place of Lemma \ref{exacther}.
\item[Case 4] Let $n\equiv 1(\mod 4).$ In this case $n-1\equiv 0(\mod 4)$ and $n-2\equiv 3(\mod 4).$ The proof is similar to Case 3.

\end{description} This completes the proof.
\end{proof} \begin{Corollary}
     Let $n\geq 2$ and $S=K[V(D_n)]$. Then 
			
			\begin{equation*}
				\pdim(S/I(D_n))=\left\{\begin{matrix}
				2n-\lceil\frac{n+1}{2} \rceil+1, & \quad \text{if}\, \, n\equiv 0,1\, (\mod\, 4);\\ \\
					2n-\lceil\frac{n+1}{2} \rceil+2, &\quad \text{if}\, \, n\equiv 2,3 \, (\mod\, 4).
				\end{matrix}\right.
			\end{equation*}
\end{Corollary} \begin{proof}
    One can get the required result by using Lemma \ref{auss13} and Lemma \ref{dotfamily}.
\end{proof}
\section{Invariants of cyclic modules associated to cubic circulant graphs}\label{CUBE}
All the cubic circulant graphs  has the form  $C_{2n}(a,n)$ with integers $1\leq a\leq n.$ Davis and Domke proved the following result:
\begin{Theorem}[{\cite{DD}}]\label{rr}
Let $ 1\leq a< n $ and $ t = \gcd(2n,a) $.
\begin{itemize}
	\item[(a)] If $ \frac{2n}{t} $ is even, then $ C_{2n}(a, n) $ is isomorphic to $ t $ copies of $ C_{\frac{2n}{t}}(1, \frac{n}{t})$. 
	\item[(b)] If $ \frac{2n}{t} $ is odd, then $ C_{2n}(a, n) $ is isomorphic to $ \frac{t}{2} $ copies of $ C_{\frac{4n}{t}}(2, \frac{2n}{t})$.
\end{itemize}
\end{Theorem} \noindent Therefore, the only connected cubic circulant graphs are those circulant graphs that are isomorphic to either  $ C_{2n}(1,n) $ for $ n\geq 2 $ or to $ C_{2n}(2, n) $ with $ n $ is odd and $n\geq 3$ (for the second circulant graph, if $n$  is not odd, then Theorem \ref{rr} implies that this circulant is not connected).   
See   Figure \ref{fig:33} for $ C_{2n}(1,n) $ and  $ C_{2n}(2, n).$ 

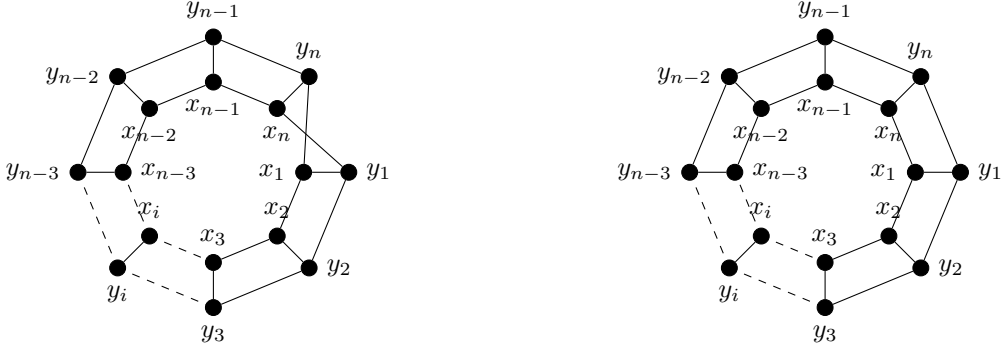
\begin{figure}[H]
	\centering
	\begin{subfigure}[b]{0.45\textwidth}
		\centering
	\[\begin{tikzpicture}[x=0.6cm, y=0.6cm]
\vertex[fill] (1) at (45:2) [label=below:${x_{n}}$]{};
\vertex[fill] (2) at (90:2) [label=below:${x_{n-1}}$]{};
\vertex[fill] (3) at (135:2) [label=below:${x_{n-2}}$]{};
\vertex[fill] (4) at (180:2) [label=right:${x_{n-3}}$]{};
\vertex[fill] (5) at (225:2) [label=above:${x_{i}}$]{};
\vertex[fill] (6) at (270:2) [label=above:${x_{3}}$]{};
\vertex[fill] (7) at (315:2) [label=above:${x_{2}}$]{};
\vertex[fill] (8) at (360:2) [label=left:${x_{1}}$]{};
\vertex[fill] (a1) at (45:3) [label=above:${y_{n}}$]{};
\vertex[fill] (a2) at (90:3) [label=above:${y_{n-1}}$]{};
\vertex[fill] (a3) at (135:3) [label=left:${y_{n-2}}$]{};
\vertex[fill] (a4) at (180:3) [label=left:${y_{n-3}}$]{};
\vertex[fill] (a5) at (225:3) [label=below:${y_{i}}$]{};
\vertex[fill] (a6) at (270:3) [label=below:${y_{3}}$]{};
\vertex[fill] (a7) at (315:3) [label=right:${y_{2}}$]{};
\vertex[fill] (a8) at (360:3) [label=right:${y_{1}}$]{};	
\draw (225:2) node {} -- (180:2) [dashed] node {};
\draw (225:3) node {} -- (270:3) [dashed] node {};
\draw (225:3) node {} -- (180:3) [dashed] node {};	
\draw (225:2) node {} -- (270:2) [dashed] node {};	
\path 
(1) edge (a8)
(7) edge (8)
(6) edge (7)
(4) edge (3)
(2) edge (3)
(1) edge (2)
(a8) edge (8)
(7) edge (a7)
(6) edge (a6)
(a4) edge (4)
(a3) edge (3)
(2) edge (a2)
(1) edge (a1)	
(a1) edge (8)
(a7) edge (a8)
(a6) edge (a7)
(a4) edge (a3)
(a2) edge (a3)
(a1) edge (a2)	
(5) edge (a5)	
;
\end{tikzpicture}\]
	\end{subfigure}
	\hfill
	\begin{subfigure}[b]{0.45\textwidth}
				\centering
	\[\begin{tikzpicture}[x=0.6cm, y=0.6cm]
\vertex[fill] (1) at (45:2) [label=below:${x_{n}}$]{};
\vertex[fill] (2) at (90:2) [label=below:${x_{n-1}}$]{};
\vertex[fill] (3) at (135:2) [label=below:${x_{n-2}}$]{};
\vertex[fill] (4) at (180:2) [label=right:${x_{n-3}}$]{};
\vertex[fill] (5) at (225:2) [label=above:${x_{i}}$]{};
\vertex[fill] (6) at (270:2) [label=above:${x_{3}}$]{};
\vertex[fill] (7) at (315:2) [label=above:${x_{2}}$]{};
\vertex[fill] (8) at (360:2) [label=left:${x_{1}}$]{};
\vertex[fill] (a1) at (45:3) [label=above:${y_{n}}$]{};
\vertex[fill] (a2) at (90:3) [label=above:${y_{n-1}}$]{};
\vertex[fill] (a3) at (135:3) [label=left:${y_{n-2}}$]{};
\vertex[fill] (a4) at (180:3) [label=left:${y_{n-3}}$]{};
\vertex[fill] (a5) at (225:3) [label=below:${y_{i}}$]{};
\vertex[fill] (a6) at (270:3) [label=below:${y_{3}}$]{};
\vertex[fill] (a7) at (315:3) [label=right:${y_{2}}$]{};
\vertex[fill] (a8) at (360:3) [label=right:${y_{1}}$]{};	
\draw (225:2) node {} -- (180:2) [dashed] node {};
\draw (225:3) node {} -- (270:3) [dashed] node {};
\draw (225:3) node {} -- (180:3) [dashed] node {};	
\draw (225:2) node {} -- (270:2) [dashed] node {};	
\path 
(1) edge (8)
(7) edge (8)
(6) edge (7)
(4) edge (3)
(2) edge (3)
(1) edge (2)
(a8) edge (8)
(7) edge (a7)
(6) edge (a6)
(a4) edge (4)
(a3) edge (3)
(2) edge (a2)
(1) edge (a1)	
(a1) edge (a8)
(a7) edge (a8)
(a6) edge (a7)
(a4) edge (a3)
(a2) edge (a3)
(a1) edge (a2)	
(5) edge (a5)	
;
\end{tikzpicture}\]
	\end{subfigure}
	\hfill
	\caption{From left to right $ C_{2n}(1,n) $ and $ C_{2n}(2,n) $.}\label{fig:33}
\end{figure}		

 \noindent  By using  Lemma \ref{LEMMA1.5} and Theorem \ref{rr},  it suffices to find the depth, projective dimension and lower bound for Stanley depth of the quotient rings of the edge ideals of $C_{2n}(1,n)$  and $C_{2n}(2,n)$ with $n$ odd. Therefore, in this section,  we first find the values of depth and projective dimension of cyclic modules $K[V(C_{2n}(1,n)]/I(C_{2n}(1,n))$  and $K[V(C_{2n}(2,n))/I(C_{2n}(2,n)).$  We give values and bounds for Stanley depth  of such modules. At the end, we compute the values of depth, projective dimension and lower bounds for Stanley depth of all cubic circulant graphs.

The following example will be helpful in understanding the  proofs of this section. 
Using Figure \ref{f3}, it is easy to see that we have the following isomorphism:
	\begin{equation*}
	   	K[V(C_{16}(1,8))]/(I(C_{16}(1,8)):x_{8}) \cong K[V(D_{5})]/I(D_5) \tensor_K K[x_8].
	\end{equation*}

 \begin{figure}[H]
			\centering
			\begin{subfigure}[b]{0.45\textwidth}
				\centering
		\[\begin{tikzpicture}[x=0.6cm, y=0.6cm]
	\vertex[fill] (1) at (45:2) [label=below:${x_{8}}$]{};
	\vertex[fill] (2) at (90:2) [label=below:${x_{7}}$]{};
	\vertex[fill] (3) at (135:2) [label=below:${x_{6}}$]{};
	\vertex[fill] (4) at (180:2) [label=right:${x_{5}}$]{};
	\vertex[fill] (5) at (225:2) [label=above:${x_{4}}$]{};
	\vertex[fill] (6) at (270:2) [label=above:${x_{3}}$]{};
	\vertex[fill] (7) at (315:2) [label=above:${x_{2}}$]{};
	\vertex[fill] (8) at (360:2) [label=left:${x_{1}}$]{};
	\vertex[fill] (a1) at (45:3) [label=above:${y_{8}}$]{};
	\vertex[fill] (a2) at (90:3) [label=above:${y_{7}}$]{};
	\vertex[fill] (a3) at (135:3) [label=left:${y_{6}}$]{};
	\vertex[fill] (a4) at (180:3) [label=left:${y_{5}}$]{};
	\vertex[fill] (a5) at (225:3) [label=below:${y_{4}}$]{};
	\vertex[fill] (a6) at (270:3) [label=below:${y_{3}}$]{};
	\vertex[fill] (a7) at (315:3) [label=right:${y_{2}}$]{};
	\vertex[fill] (a8) at (360:3) [label=right:${y_{1}}$]{};	
		
	\path 
(6) edge (5)
(a6) edge (a5)
(5) edge (4)
(a5) edge (a4)
(1) edge (a8)
(1) edge (2)
 (1) edge (a1)
	(7) edge (8)
	(6) edge (7)
	(4) edge (3)
	(2) edge (3)
	(a8) edge (8)
	(7) edge (a7)
	(6) edge (a6)
	(a4) edge (4)
	(a3) edge (3)
	(2) edge (a2)	
	(a1) edge (8)
	(a7) edge (a8)
	(a6) edge (a7)
	(a4) edge (a3)
	(a2) edge (a3)
	(a1) edge (a2)	
	(5) edge (a5)	
		;
		\end{tikzpicture}\]
				
			\end{subfigure}
			\hfill
			\begin{subfigure}[b]{0.45\textwidth}
				\centering
		 \[\begin{tikzpicture}[x=0.6cm, y=0.6cm]
\vertex[fill] (2) at (90:2) [label=below:${x_{7}}$]{};
\vertex[fill] (3) at (135:2) [label=below:${x_{6}}$]{};
\vertex[fill] (4) at (180:2) [label=right:${x_{5}}$]{};
\vertex[fill] (5) at (225:2) [label=above:${x_{4}}$]{};
\vertex[fill] (6) at (270:2) [label=above:${x_{3}}$]{};
\vertex[fill] (7) at (315:2) [label=above:${x_{2}}$]{};
\vertex[fill] (8) at (360:2) [label=left:${x_{1}}$]{};
\vertex[fill] (a1) at (45:3) [label=above:${y_{8}}$]{};
\vertex[fill] (a2) at (90:3) [label=above:${y_{7}}$]{};
\vertex[fill] (a3) at (135:3) [label=left:${y_{6}}$]{};
\vertex[fill] (a4) at (180:3) [label=left:${y_{5}}$]{};
\vertex[fill] (a5) at (225:3) [label=below:${y_{4}}$]{};
\vertex[fill] (a6) at (270:3) [label=below:${y_{3}}$]{};
\vertex[fill] (a7) at (315:3) [label=right:${y_{2}}$]{};
\vertex[fill] (a8) at (360:3) [label=right:${y_{1}}$]{};

\path 
(6) edge (5)
(a6) edge (a5)
(5) edge (4)
(a5) edge (a4)

(7) edge (8)
(6) edge (7)
(4) edge (3)
(7) edge (a7)
(6) edge (a6)
(a4) edge (4)
(a3) edge (3)
(a6) edge (a7)
(a4) edge (a3)
(a2) edge (a3)	
(5) edge (a5)	
;
		\end{tikzpicture}\]
			\end{subfigure}
			\caption{From left to right $G_{(I(C_{16}(1,8)))}$ and $G_{(I(C_{16}(1,8)): x_{8})}.$}
			\label{f3}
		\end{figure}
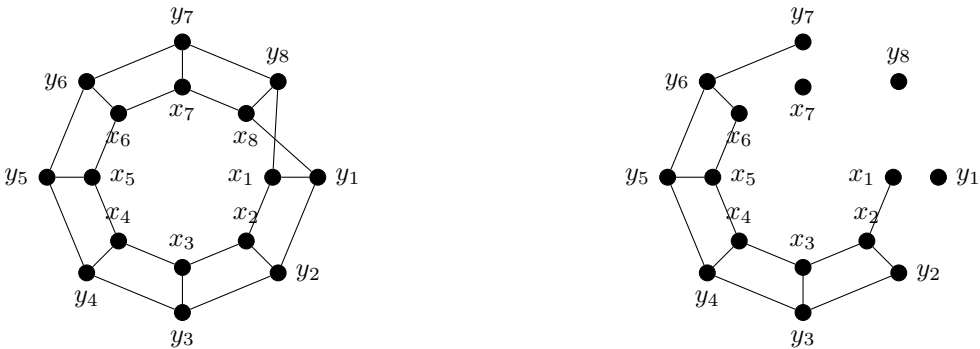 	
\begin{Proposition}\label{mobious}
    For $n\geq 2,$ let $ G=C_{2n}(1,n) $ and  $S=K[V(G)].$  Then
\begin{equation*}
				\depth{(S/I(G))}=\left\{\begin{matrix}
					\lceil\frac{n}{2}\rceil, & \quad \quad \quad \text{if}\, \, n\equiv 1\,(\mod 4);\\\\
					\lceil\frac{n-1}{2}\rceil, &    \text{otherwise}.
				\end{matrix}\right.
			\end{equation*}
\end{Proposition}
\begin{proof}  If $n=2,$ we have $C_4(1,2)\cong \mathbb{K}_4,$ then by Lemma \ref{com} the required result follows. If $n=3,$ one can see that the required result holds by using Macaulay2 \cite{macaulay}.   If $n=4,$ we consider the following short exact sequence
\begin{equation}\label{es2qee}
0\longrightarrow (I(G):y_{1})/I(G)\xrightarrow{\cdot y_{1}} S/I(G)\longrightarrow S/(I(G):y_{1})\longrightarrow 0.
\end{equation} Here 
\begin{equation}\label{gg1eehh}
     K[V(G)]/(I(G):y_1)\cong  \frac{K[x_2,x_3,y_3,y_4]}{(x_2x_3,x_3y_3,y_3y_4)}[y_1]\cong K[V(\mathbb{P}_{4})]/I(\mathbb{P}_{4})\tensor_K K[y_{1}],
 \end{equation} and  $N_{G}(y_1)=\{x_{4},x_1,y_2\},$ $S_1=K[V(G)\backslash N_{G}(x_{4})],$ $S_2=K[V(G)\backslash (N_{G}(x_{1})\cup \{x_{4}\})],$ $S_3=K[V(G)\backslash (N_{G}(y_{2})\cup \{x_{4},x_1\})],$ $J_1=(S_1\cap I(G)),$ $J_2=(S_2\cap I(G)),$ $J_3=(S_3\cap I(G)),$ then by using Lemma \ref{lemiso}, we have  
 \begin{equation}\label{gg2eehh}
 \begin{split}
     (I(G):y_{1})/I(G) &\cong S_1/J_1[x_{4}] \oplus S_2/J_2[x_{1}] \oplus S_3/J_3[y_{2}]\\&\cong \frac{K[x_1,x_2,y_2,y_3]}{(x_1x_2,x_2y_2,y_2y_3)}[x_4] \oplus \frac{K[x_3,y_2,y_3]}{(x_3y_3,y_3y_2)}[x_1] \oplus \frac{K[x_3,y_4]}{(0)}[y_2]\\&\cong K[V(\mathbb{P}_{4})]/I(\mathbb{P}_{4})\tensor_K K[x_{4}]\oplus \big(K[V(\mathbb{P}_{3})]/I(\mathbb{P}_{3})\tensor_K K[x_{1}]\big)\oplus K[x_3,y_4,y_2].\end{split} \end{equation}  By applying Lemma \ref{111} and Lemma \ref{paath} on Equations (\ref{gg1eehh}) and  (\ref{gg2eehh}),  we get $\depth(K[V(G)]/(I(G):y_{1}))=\depth(K[V(\mathbb{P}_{4})]/I(\mathbb{P}_{4}))+\depth(K[y_1])=3$  and \begin{equation*}
         \begin{split}
             \depth((I(G):y_{1})/I(G))&=\min\{\depth(K[V(\mathbb{P}_{4})]/I(\mathbb{P}_{4}))+1,\depth(K[V(\mathbb{P}_{3})]/I(\mathbb{P}_{3}))+1,\\& \quad\quad\quad \quad\depth(K[x_3,y_4,y_2])\}=2.
         \end{split}
     \end{equation*}
 By Depth Lemma on Equation (\ref{es2qee}), we have $\depth(S/I(G))=2.$  Let $n\geq 5.$  We have the following $K$-algebra isomorphisms:
\begin{equation}\label{w3}
    S/(I(G):y_{1})\cong K[V(D_{n-3})]/I(D_{n-3})\tensor_K K[y_{1}],
\end{equation} and if $N_{G}(y_1)=\{x_{n},x_1,y_2\},$ $S_1=K[V(G)\backslash N_{G}(x_{n})],$ $S_2=K[V(G)\backslash (N_{G}(x_{1})\cup \{x_{n}\})],$ $S_3=K[V(G)\backslash (N_{G}(y_{2})\cup \{x_{n},x_1\})],$ $J_1=(S_1\cap I(G)),$ $J_2=(S_2\cap I(G)),$ $J_3=(S_3\cap I(G)),$ then by using Lemma \ref{lemiso}, we get  \begin{equation}\label{w4} \begin{split}
				        (I(G):y_{1})/I(G) &\cong S_1/J_1[x_{n}] \oplus S_2/J_2[x_{1}] \oplus S_3/J_3[y_{2}]\\& \cong  \frac{K[x_{1},\dots,x_{n-2},y_{2},\dots,y_{n-1}]}{\big(\cup^{n-3}_{i=2}\{x_{i}y_{i},x_{i}x_{i+1},y_{i}y_{i+1}\}\cup\{x_{n-2}y_{n-2},x_1x_2,y_{n-2}y_{n-1} \}\big)}[x_{n}]
\\ &\quad \oplus \frac{K[x_{3},\dots,x_{n-1},y_{2},\dots,y_{n-1}]}{\big(\cup^{n-2}_{i=3}\{x_{i}y_{i},x_{i}x_{i+1},y_{i}y_{i+1}\}\cup\{x_{n-1}y_{n-1},y_2y_3\}\big)}[x_{1}]\\ &\quad \oplus \frac{K[x_{3},\dots,x_{n-1},y_{4},\dots,y_{n}]}{\big(\cup^{n-2}_{i=4}\{x_{i}y_{i},x_{i}x_{i+1},y_{i}y_{i+1}\}\cup\{x_{n-1}y_{n-1},x_{3}x_{4},y_{n-1}y_n\}\big)}[y_{2}]\\\\&\cong  K[V(D_{n-3})]/I(D_{n-3}) \tensor_K K[x_n] \oplus K[V(B_{n-3})]/I(B_{n-3})  \tensor_K K[x_1] \\& \quad  \oplus  K[V(D_{n-4})]/I(D_{n-4}) \tensor_K K[y_2].
				    \end{split}
				\end{equation}  By using Equations (\ref{w3}),  (\ref{w4}) and Lemma \ref{111}, we have \begin{equation}\label{u12a}
    \depth{( S/(I(G):y_{1}))}=\depth{K[V(D_{n-3})]/I(D_{n-3})}+\depth{ K[y_{1}]},
\end{equation} 
\begin{equation}\label{u13a}
\begin{split}
    \depth\big((I(G):y_{1})/I(G)\big) & =\min\Big\{\depth{K[V(D_{n-3})]/I(D_{n-3})}+1, \depth{K[V(B_{n-3})]/I(B_{n-3})}+1, \\& \quad \quad \quad  \quad \quad \quad \depth{K[V(D_{n-4})]/I(D_{n-4})}+1\Big\}.
\end{split}
\end{equation} 
 Now, if $n\equiv 1\,(\mod 4),$ then $n-3\equiv 2\,(\mod 4)$  and  $n-4\equiv 1\,(\mod 4).$  By using Lemma \ref{dotfamily} in Equation (\ref{u12a}), we get
\begin{equation*}
    \depth{( S/(I(G):y_{1}))}=\depth{K[V(D_{n-3})]/I(D_{n-3})}+1=\lceil\frac{n-3+1}{2}\rceil +1=\lceil\frac{n}{2}\rceil.
\end{equation*} 
    By applying Lemma \ref{The1}, Lemma \ref{dotfamily} and Remark \ref{rem1} on Equation (\ref{u13a}), we get 
    \begin{equation*}
        \begin{split}
            \depth\big((I(G):y_{1})/I(G)\big)&=\min\Big\{\lceil\frac{n-3+1}{2}\rceil +1, \lceil\frac{n-3+1}{2}\rceil +1,\lceil\frac{n-4+1}{2}\rceil +1+1\Big\}\\ &= \min\Big\{\lceil\frac{n}{2}\rceil , \lceil\frac{n}{2}\rceil ,\lceil\frac{n+1}{2}\rceil \Big\}\\&=\lceil\frac{n}{2}\rceil.
        \end{split}
    \end{equation*} As $ \depth{( S/(I(G):y_{1}))}=\depth\big((I(G):y_{1})/I(G)\big)=\lceil\frac{n}{2}\rceil,$ therefore  by applying  Depth Lemma on Equation (\ref{es2qee}), we get $\depth(S/I(G)=\lceil\frac{n}{2}\rceil.$ 	  If $n\equiv 2\,(\mod 4),$ then $n-3\equiv 3\,(\mod 4)$  and  $n-4\equiv 2\,(\mod 4).$ To prove the result, we use  similar strategy of the previous case and get the required result that is  $\depth{(S/I(G))}=\lceil\frac{n-1}{2}\rceil.$  If    $n\equiv 0, 3\,(\mod 4),$ the proof is similar. 
\end{proof}
\begin{Corollary} 
		  For $n\geq 2,$ let $ G=C_{2n}(1,n) $ and  $S=K[V(G)].$ Then
    
\begin{equation*}
				\pdim{(S/I(G))}=\left\{\begin{matrix}
					2n-\lceil\frac{n}{2}\rceil, & \quad \quad \quad \text{if}\, \, n\equiv 1\,(\mod 4);\\\\
					2n-\lceil\frac{n-1}{2}\rceil, & \quad  \text{otherwise}.
				\end{matrix}\right.
			\end{equation*} 
		\end{Corollary}
		\begin{proof}
			The required result follows by using Lemma \ref{auss13} and Proposition \ref{mobious}.
		\end{proof}
\begin{Proposition}\label{mobiousstan}
    For $n\geq 2,$ let $ G=C_{2n}(1,n) $ and  $S=K[V(G)].$  Then

\begin{equation*}
				\sdepth{(S/I(G))}=\left\{\begin{matrix}
					\lceil\frac{n}{2}\rceil, & \quad \text{if}\, \, n\equiv 1\,(\mod 4);\\\\
					\lceil\frac{n-1}{2}\rceil, & \quad \text{if}\, \, n\equiv 2\,(\mod 4).
				\end{matrix}\right.
			\end{equation*}
     If  $ n\equiv 0,3\,(\mod 4),$ then
      $$\lceil\frac{n-1}{2}\rceil\leq \sdepth{(S/I(G))}\leq \lceil\frac{n}{2}\rceil+1.$$ 
\end{Proposition}
\begin{proof}  If $n=2,$ we have $C_4(1,2)\cong \mathbb{K}_4,$ then the  result follows by Lemma \ref{com} that is $\sdepth{(S/I(G))}\geq 1.$ If $n=3,$  we find the required lower bound by using Lemma \ref{stan1}. For the upper bound, since $y_1\notin I(G),$ by Lemma \ref{Pro7}, we have $\sdepth(S/I(G))\leq \sdepth(S/(I(G):y_{1}).$  Here $$K[V(G)]/(I(G):y_1)\cong  \frac{K[x_2,y_3]}{(0)}[y_1],$$ and by Lemma \ref{111},  $\sdepth(S/I(G))\leq \sdepth(K[x_1,x_3,y_2]) =3.$ If $n=4,$  one can find a lower bound for Stanley depth in a similar way as depth in Proposition \ref{mobious} just by using  Lemma \ref{le1} in place of  Depth Lemma, that is $\sdepth(S/I(G))\geq 2.$ For the upper bound, by using Equation (\ref{gg1eehh}), Lemma \ref{Pro7} and Lemma \ref{111}, we have  $\sdepth(S/I(G))\leq \sdepth(S/(I(G):y_{1})=\sdepth(K[V(\mathbb{P}_{4})]/I(\mathbb{P}_{4}))+1 =3.$  Let $n\geq 5.$ By using a similar strategy of depth as in Proposition \ref{mobious} and applying Lemma \ref{le1} in place of Depth Lemma on Equation (\ref{es2qee}), we get the required lower bound for Stanley depth.  For other inequality, by using Lemma \ref{Pro7}, Lemma \ref{111} and Lemma \ref{dotfamily} on Equation (\ref{w3}), we get the required result. This completes the proof.
\end{proof}

 \begin{Proposition}\label{circular}
    For $n\geq 3,$  let $ G=C_{2n}(2,n) $  and $S=K[V(G)].$ Then 
\begin{equation*}
				\depth{(S/I(G))}=\left\{\begin{matrix}
					\lceil\frac{n-1}{2}\rceil, & \quad \quad \text{if}\, \, n\equiv 1\,(\mod 4);\\\\
					\lceil\frac{n}{2}\rceil, &    \text{otherwise}.
				\end{matrix}\right.
			\end{equation*}
\end{Proposition} \begin{proof}  If $n=3,$ one can easily see that the result holds by  Macaulay2 \cite{macaulay}. If $n=4,$ we  consider the following short exact sequence
\begin{equation}\label{es2qeer}
0\longrightarrow (I(G):y_{4})/I(G)\xrightarrow{\cdot y_{4}} S/I(G)\longrightarrow S/(I(G):y_{4})\longrightarrow 0. 
\end{equation}
We have 
\begin{equation}\label{gg1eer}
     K[V(G)]/(I(G):y_4)\cong  \frac{K[x_1,x_2,x_3,y_2]}{(x_1x_2,x_2x_3,x_2y_2)}[y_4]\cong K[V(\mathbb{S}_4)]/I(\mathbb{S}_4)\tensor_K K[y_4],
 \end{equation} and if we have   $N_{G}(y_4)=\{y_{3},x_4,y_1\},$ $S_1=K[V(G)\backslash N_{G}(y_{3})],$ $S_2=K[V(G)\backslash (N_{G}(x_{4})\cup \{y_{3}\})],$ $S_3=K[V(G)\backslash (N_{G}(y_{1})\cup \{y_3, x_{4}\})],$ $J_1=(S_1\cap I(G)),$ $J_2=(S_2\cap I(G)),$ $J_3=(S_3\cap I(G)),$ then by using Lemma \ref{lemiso}, we get 
 \begin{equation}\label{gg2eer}\begin{split}
      (I(G):y_{4})/I(G)  &\cong S_1/J_1[y_{3}] \oplus S_2/J_2[x_{4}] \oplus S_3/J_3[y_{1}]\\&\cong \frac{K[x_1,x_2,x_4,y_1]}{(x_1x_2,x_1x_4,x_1y_1)} [y_3] \oplus \frac{K[x_2,y_1,y_2]}{(y_1y_2,y_2x_2)}[x_4] \oplus \frac{K[x_2,x_3]}{(x_2x_3)} [y_1]\\& \cong   K[V(\mathbb{S}_4)]/I(\mathbb{S}_4)\tensor_K K[y_3] \oplus \big(K[V(\mathbb{P}_3)]/I(\mathbb{P}_3)\tensor_K K[x_4]\big)\\&  \quad \oplus  K[V(\mathbb{P}_2)]/I(\mathbb{P}_2)\tensor_K K[y_1 ].
 \end{split}\end{equation}  We apply Lemma \ref{111}, Lemma \ref{paath} and Lemma \ref{leAli} on Equation (\ref{gg1eer}), we get $\depth(K[V(G)]/(I(G):y_{4}))=\depth( K[V(\mathbb{S}_4)]/I(\mathbb{S}_4))+1=2$  and by Equation (\ref{gg2eer})
 \begin{equation*}\begin{split}
       \depth((I(G):y_{4})/I(G))&=\min\{\depth(K[V(\mathbb{S}_4)]/I(\mathbb{S}_4))+1,\depth(K[V(\mathbb{P}_3)]/I(\mathbb{P}_3))+1, \\& \quad\quad \quad\quad \depth(K[V(\mathbb{P}_2)]/I(\mathbb{P}_2))+1\}\\&=2.
 \end{split}\end{equation*}
 By using Depth Lemma on Equation (\ref{es2qeer}), $\depth(S/I(G))=2.$ Let $n\geq 5.$ Consider the following  short exact sequence
\begin{equation}\label{u11r}
0\longrightarrow (I(G):y_{n})/I(G)\xrightarrow{\cdot y_{n}} S/I(G)\longrightarrow S/(I(G):y_{n})\longrightarrow 0.
\end{equation}   Here
\begin{equation}\label{j22r}
    S/(I(G):y_{n})\cong K[V(C_{n-3})]/I(C_{n-3})\tensor_K K[y_{n}], 
\end{equation} and  if   $N_{G}(y_n)=\{y_{n-1},x_n,y_1\},$ $S_1=K[V(G)\backslash N_{G}(y_{n-1})],$ $S_2=K[V(G)\backslash (N_{G}(x_{n})\cup \{y_{n-1}\})],$ $S_3=K[V(G)\backslash (N_{G}(y_{1})\cup \{y_{n-1}, x_{n}\})],$ $J_1=(S_1\cap I(G)),$ $J_2=(S_2\cap I(G)),$ $J_3=(S_3\cap I(G)),$ then by using Lemma \ref{lemiso}, we have   \begin{equation} \begin{split}\label{nnnr}
				        (I(G):y_{n})/I(G)  &\cong S_1/J_1[y_{n-1}] \oplus S_2/J_2[x_{n}] \oplus S_3/J_3[y_{1}]\\& \cong  \frac{K[x_{1},\dots,x_{n-2},x_n,y_{1},\dots,y_{n-3}]}{\big(\cup^{n-4}_{i=1}\{x_{i}y_{i},x_{i}x_{i+1},y_{i}y_{i+1}\}\cup\{x_{n-3}y_{n-3},x_1x_n,x_{n-3}x_{n-2}\} \big)}[y_{n-1}]
\\ &\quad \oplus \frac{K[x_{2},\dots,x_{n-2},y_{1},\dots,y_{n-2}]}{\big(\cup^{n-3}_{i=2}\{x_{i}y_{i},x_{i}x_{i+1},y_{i}y_{i+1}\}\cup\{x_{n-2}y_{n-2},y_1y_2\}\big)}[x_{n}]\\ &\quad \oplus \frac{K[x_{2},\dots,x_{n-1},y_{3},\dots,y_{n-2}]}{\big(\cup^{n-3}_{i=2}\{x_{i}y_{i},x_{i}x_{i+1},y_{i}y_{i+1}\}\cup\{x_{n-2}y_{n-2},x_{n-2}x_{n-1},x_2x_3\}\big)}[y_{1}]\\\\&\cong K[V(C_{n-3})]/I(C_{n-3})  \tensor_K K[y_{n-1}] \oplus K[V(B_{n-3})]/I(B_{n-3})  \tensor_K K[x_n]\\& \quad \oplus K[V(C_{n-4})]/I(C_{n-4})  \tensor_K K[y_1]
				    \end{split}
				\end{equation}   By Equations (\ref{nnnr}),  (\ref{j22r})  and Lemma \ref{111}, we get \begin{equation}\label{u12}
    \depth{( S/(I(G):y_{n}))}=\depth{K[V(C_{n-3})]/I(C_{n-3})}+\depth{ K[y_{n}]},
\end{equation} and 

\begin{equation}\label{u13}
\begin{split}
    \depth\big((I(G):y_{n})/I(G)\big)&=\min\Big\{\depth{K[V(C_{n-3})]/I(C_{n-3})}+1, \depth{K[V(B_{n-3})]/I(B_{n-3})}+1, \\ & \quad \quad \quad \quad\quad \quad\depth{K[V(C_{n-4})]/I(C_{n-4})}+1\Big\}.
\end{split}
\end{equation} If $n\equiv 1\,(\mod 4),$ then $n-3\equiv 2\,(\mod 4)$  and  $n-4\equiv 1\,(\mod 4).$  By applying Lemma \ref{diamond}, Lemma \ref{111} on Equation (\ref{u12}), we get
\begin{equation}
    \depth{( S/(I(G):y_{n}))}=\lceil\frac{n-3+1}{2}\rceil +1+1=\lceil\frac{n}{2}\rceil+1.
\end{equation} 
    By applying Lemma \ref{The1}, Lemma \ref{diamond} and Remark \ref{rem1}  on Equation (\ref{u13}), we get 
    \begin{equation}
        \begin{split}
            \depth\big((I(G):y_{n})/I(G)\big)&=\min\Big\{\lceil\frac{n-3+1}{2}\rceil +1+1, \lceil\frac{n-3+1}{2}\rceil +1,\lceil\frac{n-4+1}{2}\rceil +1\Big\}\\ &= \min\Big\{\lceil\frac{n}{2}\rceil +1, \lceil\frac{n}{2}\rceil ,\lceil\frac{n-1}{2}\rceil \Big\}\\&=\lceil\frac{n-1}{2}\rceil.
        \end{split}
    \end{equation} Since, $\depth\big((I(G):y_{n})/I(G)\big)< \depth{( S/(I(G):y_{n}))},$  thus the required result follows by applying Depth Lemma on Equation (\ref{u11r}). Similarly, if   $n\equiv 2\,(\mod 4),$ then $n-3\equiv 3\,(\mod 4)$  and  $n-4\equiv 2\,(\mod 4),$ then the required results follows by similar strategy that is $\depth{(S/I(G))}=\lceil\frac{n}{2}\rceil.$
    If  $ n\equiv 0,3\,(\mod 4),$  the proof follows in a similar way. 
\end{proof} 
\begin{Corollary} For $n\geq 3,$  let $ G=C_{2n}(2,n) $  and $S=K[V(G)].$ Then

\begin{equation*}
				\pdim{(S/I(G))}=\left\{\begin{matrix}
				2n	-\lceil\frac{n-1}{2}\rceil, & \quad \quad\text{if}\, \, n\equiv 1\,(\mod 4);\\\\
				2n-	\lceil\frac{n}{2}\rceil, &    \text{otherwise}.
				\end{matrix}\right.
			\end{equation*}
		\end{Corollary}
		\begin{proof}
			The proof follows by  Lemma \ref{auss13} and Proposition  \ref{circular}.
		\end{proof}
  \begin{Proposition}\label{circularstan}
    For $n\geq 3,$  let $ G=C_{2n}(2,n) $  and $S=K[V(G)].$  If  $ n\equiv 0,3\,(\mod 4),$ then $$\sdepth{(S/I(G))}  =\lceil\frac{n}{2}\rceil.$$ If $n\equiv 1\,(\mod 4),$ then $$ \lceil\frac{n-1}{2}\rceil \leq \sdepth{(S/I(G))} \leq \lceil\frac{n}{2}\rceil+1,$$ and if $n\equiv 2\,(\mod 4),$ we have $$ \lceil\frac{n-1}{2}\rceil \leq \sdepth{(S/I(G))} \leq \lceil\frac{n-1}{2}\rceil+1.$$ 
\end{Proposition} \begin{proof}  If $n=3,$ one can get lower bound  by using CoCoA \cite{cocoa} that is $\sdepth(S/I(G))\geq 2.$  For the upper bound, by Lemma \ref{Pro7}, we have $\sdepth(S/I(G))\leq \sdepth(S/(I(G):y_{3}).$ Here
\begin{equation}\label{gg1eerr}
     K[V(G)]/(I(G):y_3)\cong  \frac{K[x_1,x_2]}{(x_1x_2)}[y_3],
 \end{equation}  by Lemma \ref{111} and Lemma \ref{paath}, we get $\sdepth(S/I(G))\leq 2.$ Let $n=4.$  For the upper bound, since $y_4\notin I(G),$ by Lemma \ref{Pro7}, we have $\sdepth(S/I(G))\leq \sdepth(S/(I(G):y_{4}).$  Here $$K[V(G)]/(I(G):y_4)\cong  \frac{K[x_1,x_2,x_3,y_2]}{(x_1x_2,x_2y_2,x_2x_3)}[y_4]\cong K[V(\mathbb{S}_4)]/I(\mathbb{S}_4)\tensor_K K[y_4],$$ and by Lemma \ref{111} and Lemma \ref{leAli},  $\sdepth(S/I(G))\leq \sdepth(K[V(\mathbb{S}_4)]/I(\mathbb{S}_4))+1 =2.$ For other inequality, one can find Stanley depth in a similar way as depth in Proposition \ref{circular} just by using  Lemma \ref{le1} in place of  Depth Lemma, that is $\sdepth(S/I(G))\geq 2.$  Let $n\geq 5.$  We get the required lower bound for Stanley depth by using a similar strategy of depth as in Proposition \ref{circular} and applying Lemma \ref{le1} in-place of Depth Lemma on Equation (\ref{u11r}).   For other inequality, by using Lemma \ref{Pro7}, Lemma \ref{111} and Lemma \ref{diamond}  on Equation (\ref{j22r}), the required result follows.\end{proof}
\noindent Before proving Theorem \ref{cy1},  we make the following remark.
\begin{Remark}\label{ww}
  \em{Let  $n\geq 2,$ $ t = \gcd(2n,a) $ and  $ 1\leq a< n .$  Note that by Theorem \ref{rr}(b),  $C_{2n}(a,n)$ is isomorphic to $\frac{t}{2}$ copies of $C_{\frac{4n}{t}}(2,\frac{2n}{t}).$  For $C_{\frac{4n}{t}}(2, \frac{2n}{t}),$ we only need to consider the case when $\frac{2n}{t}$ is odd. If $\frac{2n}{t}$ is even, then by Theorem \ref{rr}}(a), we have $t$ disjoint copies of $C_{\frac{2n}{t}}(1, \frac{n}{t}).$
\end{Remark}
\begin{Theorem}\label{cy1}
 Let $n\geq 2,$ $ t = \gcd(2n,a) $ and  $ 1\leq a< n .$ \begin{itemize}
	\item[(a)] If $ \frac{2n}{t} $ is even, then 

\begin{equation*}
				\depth{(K[V(C_{2n}(a,n))]/I(C_{2n}(a,n)))}=\left\{\begin{matrix}
				\lceil\frac{n}{2t}\rceil \cdot t	, & \quad\quad \text{if}\, \, \frac{n}{t}\equiv 1\,(\mod 4) ;\\\\
					\lceil\frac{n-t}{2t}\rceil \cdot t, &  \text{otherwise}.
				\end{matrix}\right.
			\end{equation*}
 \item[(b)] If $ \frac{2n}{t} $ is odd, then 

\begin{equation*}
				\depth{(K[V(C_{2n}(a,n))]/I(C_{2n}(a,n)))}=\left\{\begin{matrix}
					\lceil\frac{2n-t}{2t}\rceil \cdot \frac{t}{2}, & \quad \text{if}\, \,\frac{2n}{t}\equiv 1\,(\mod 4) ;\\\\
					\lceil\frac{n}{t}\rceil \cdot \frac{t}{2}, & \quad \text{if}\, \,\frac{2n}{t}\equiv 3\,(\mod 4).
				\end{matrix}\right.
			\end{equation*}
 \end{itemize}
\end{Theorem}
\begin{proof}
Let  $ \frac{2n}{t} $ is even. Since $t = \gcd(2n,a),$ therefore $ \frac{n}{t}\geq 2$ and a positive integer. Now by using Proposition \ref{mobious}, we have 

\begin{equation*}
				\depth{\big(K[V(C_{\frac{2n}{t}}(1,\frac{n}{t}))]/I(C_{\frac{2n}{t}}(1,\frac{n}{t}))\big)}=\left\{\begin{matrix}
				\big\lceil\frac{\frac{n}{t}}{2}\big\rceil 	, & \quad \quad \text{if}\, \, \frac{n}{t}\equiv 1\,(\mod 4) ;\\\\
					\big\lceil\frac{\frac{n}{t}-1}{2}\big\rceil , &  \text{otherwise}.
				\end{matrix}\right.
			\end{equation*} By Theorem \ref{rr}, $C_{2n}(a,n)$ is isomorphic to $t$ copies of $C_{\frac{2n}{t}}(1,\frac{n}{t}).$ Therefore, by  Lemma \ref{LEMMA1.5}, 

\begin{equation*}
				\depth{(K[V(C_{2n}(a,n))]/I(C_{2n}(a,n)))}=\left\{\begin{matrix}
				\big\lceil\frac{n}{2t}\big\rceil \cdot  t	, &\quad \quad \text{if}\, \, \frac{n}{t}\equiv 1\,(\mod 4) ;\\\\
					\big\lceil\frac{\frac{n}{t}-1}{2}\big\rceil \cdot t, & \text{otherwise}.
				\end{matrix}\right.
			\end{equation*} Now, if  $ \frac{2n}{t} $ is odd, then $ \frac{2n}{t}> 2$ and a positive integer. By using a similar strategy, use Proposition  \ref{circular} in place of  Proposition \ref{mobious} and by Theorem \ref{rr},  $C_{2n}(a,n)$ is isomorphic to $\frac{t}{2}$ copies of $C_{\frac{4n}{t}}(2,\frac{2n}{t}).$ By Remark \ref{ww}, it is enough to consider the cases when  $\frac{2n}{t}$ is odd, therefore by  Lemma \ref{LEMMA1.5},  the required result follows 
\begin{equation*}
				\depth{(K[V(C_{2n}(a,n))]/I(C_{2n}(a,n)))}=\left\{\begin{matrix}
					\lceil\frac{\frac{2n}{t}-1}{2}\rceil \cdot \frac{t}{2}, & \quad \quad \text{if}\, \, \frac{2n}{t}\equiv 1\,(\mod 4) ;\\\\
					\lceil\frac{\frac{2n}{t}}{2}\rceil \cdot \frac{t}{2}, & \quad \quad \text{if}\, \, \frac{2n}{t}\equiv 3\,(\mod 4).
				\end{matrix}\right.
			\end{equation*}
\end{proof}
\begin{Corollary}\label{cy2}
 Let $n\geq 2,$ $ t = \gcd(2n,a) $ and  $ 1\leq a< n .$ \begin{itemize}
	\item[(a)] If $ \frac{2n}{t} $ is even, then 

\begin{equation*}
				\pdim(K[V(C_{2n}(a,n))]/I(C_{2n}(a,n)))=\left\{\begin{matrix}
				2n-\lceil\frac{n}{2t}\rceil \cdot t	, & \quad \quad \text{if}\, \, \frac{n}{t}\equiv 1\,(\mod 4) ;\\\\
					2n-\lceil\frac{n-t}{2t}\rceil \cdot t, &  \text{otherwise}.
				\end{matrix}\right.
			\end{equation*}
 \item[(b)] If $ \frac{2n}{t} $ is odd, then 

\begin{equation*}
				\pdim(K[V(C_{2n}(a,n))]/I(C_{2n}(a,n)))=\left\{\begin{matrix}
					2n-\lceil\frac{2n-t}{2t}\rceil \cdot \frac{t}{2}, & \quad \quad \text{if}\, \, \frac{2n}{t}\equiv 1\,(\mod 4) ;\\\\
					2n-\lceil\frac{n}{t}\rceil \cdot \frac{t}{2}, & \quad \quad \text{if}\, \, \frac{2n}{t}\equiv 3\,(\mod 4).
				\end{matrix}\right.
			\end{equation*}
 \end{itemize}
\end{Corollary}
\begin{proof}
    The proof follows by Lemma \ref{auss13} and Theorem \ref{cy1}.
\end{proof}

\begin{Theorem}\label{cy3}
    
 Let $n\geq 2,$ $ t = \gcd(2n,a) $ and  $ 1\leq a< n .$ \begin{itemize}
	\item[(a)] If $ \frac{2n}{t} $ is even, then 

\begin{equation*}
				\sdepth{(K[V(C_{2n}(a,n))]/I(C_{2n}(a,n)))}\geq \left\{\begin{matrix}
				\lceil\frac{n}{2t}\rceil \cdot t, & \quad \quad \text{if}\, \, \frac{n}{t}\equiv 1\,(\mod 4) ;\\\\
					\lceil\frac{n-t}{2t}\rceil \cdot t, & \text{otherwise}.
				\end{matrix}\right.
			\end{equation*}
 \item[(b)] If $ \frac{2n}{t} $ is odd, then 

\begin{equation*}
				\sdepth{(K[V(C_{2n}(a,n))]/I(C_{2n}(a,n)))}\geq\left\{\begin{matrix}
					\lceil\frac{2n-t}{2t}\rceil \cdot \frac{t}{2}, & \quad \quad \text{if}\, \, \frac{2n}{t}\equiv 1\,(\mod 4) ;\\\\
					\lceil\frac{n}{t}\rceil \cdot \frac{t}{2}, & \quad \quad \text{if}\, \, \frac{2n}{t}\equiv 3\,(\mod 4).
				\end{matrix}\right.
			\end{equation*}
 \end{itemize}
\end{Theorem} \begin{proof}

Since $t=\gcd(2n,a),$ therefore if $\frac{2n}{t}$ is even, then $\frac{n}{t}$ is a positive integer  greater than or equals to 2. By 
Proposition \ref{mobiousstan}, we have 
  
\begin{equation*}
				\sdepth{(K[V(C_{\frac{2n}{t}}(1,\frac{n}{t}))]/I(C_{\frac{2n}{t}}(1,\frac{n}{t})))}\geq \left\{\begin{matrix}
				\lceil\frac{\frac{n}{t}}{2}\rceil, & \quad \quad \text{if}\, \, \frac{n}{t}\equiv 1\,(\mod 4) ;\\\\
					\lceil\frac{\frac{n}{t}-1}{2}\rceil , & \text{otherwise}.
				\end{matrix}\right.
			\end{equation*}
By using Theorem \ref{rr}, $C_{2n}(a,n)$ is isomorphic to $t$ copies of $C_{\frac{2n}{t}}(1,\frac{n}{t}),$ therefore by  Lemma \ref{LEMMA1.5}, 
   \begin{equation*}
				\sdepth{(K[V(C_{2n}(a,n))]/I(C_{2n}(a,n)))}\geq \left\{\begin{matrix}
				\lceil\frac{\frac{n}{t}}{2}\rceil \cdot t, &\quad \quad \text{if}\, \, \frac{n}{t}\equiv 1\,(\mod 4) ;\\\\
					\lceil\frac{\frac{n}{t}-1}{2}\rceil \cdot t, &    \text{otherwise}.
				\end{matrix}\right.
			\end{equation*} Similarly, if $\frac{2n}{t}$ is odd then  $\frac{2n}{t}$ is a positive integer strictly greater than 2. We get the required result just by replacing Proposition \ref{mobiousstan} with Proposition \ref{circularstan} and by Theorem \ref{rr},  $C_{2n}(a,n)$ is isomorphic to $\frac{t}{2}$ copies of $C_{\frac{4n}{t}}(2,\frac{2n}{t}). $ By using Remark \ref{ww}, it is enough to cater the cases when  $\frac{2n}{t}$ is odd, therefore by  Lemma \ref{LEMMA1.5},  we get
\begin{equation*}
				\sdepth{(K[V(C_{2n}(a,n))]/I(C_{2n}(a,n)))}\geq\left\{\begin{matrix}
					\lceil\frac{\frac{2n}{t}-1}{2}\rceil \cdot \frac{t}{2}, & \quad \quad \text{if}\, \, \frac{2n}{t}\equiv 1\,(\mod 4) ;\\\\
					\lceil\frac{\frac{2n}{t}}{2}\rceil \cdot \frac{t}{2}, & \quad \quad \text{if}\, \, \frac{2n}{t}\equiv 3\,(\mod 4).
				\end{matrix}\right.
			\end{equation*}
\end{proof}
\begin{Remark}
  \em{ Let $n\geq 2,$ $ t = \gcd(2n,a) $ and  $ 1\leq a< n .$ Then    Stanley’s inequality holds for $K[V(C_{2n}(a,n))]/I(C_{2n}(a,n)).$}
\end{Remark}


\begin{thebibliography}{999}
%
%
%
%
%
\bibitem {AA} Alipour, A., and Tehranian, A. (2017).  Depth and Stanley depth of edge ideals of star graphs. International Journal of Applied Mathematics and Statics, 56(4), 63-69.
\bibitem{isograph} Alspach, B.,  Parsons, T. D. (1979). Isomorphism of circulant graphs and digraphs. Discrete Mathematics, 25(2), 97-108.
\bibitem{net} Bermond, J. C., Comellas, F.,  Hsu, D. F. (1995). Distributed loop computer-networks: a survey. Journal of parallel and distributed computing, 24(1), 2-10.

\bibitem{regnpro} Bouchat, R. R. (2010). Free resolutions of some edge ideals of simple graphs. Journal of Commutative Algebra, 2(1), 1-35.
\bibitem{wellcover} Brown, J.,  Hoshino, R. (2011). Well-covered circulant graphs. Discrete Mathematics, 311(4), 244-251.
\bibitem{depth} 
Bruns, W.,  Herzog, H. J. (1998). Cohen-Macaulay rings. Cambridge University Press.


\bibitem{regul} Caviglia, G., Hà, H. T., Herzog, J., Kummini, M., Terai, N.,  Trung, N. V. (2019). Depth and regularity modulo a principal ideal. \emph{Journal of Algebraic Combinatorics}, 49(1), 1-20.
\bibitem{miii} Cimpoeas, M. (2008). Some remarks on the Stanley depth for multigraded modules, Le. Matematiche LXIII, 165-171.
\bibitem{MC} Cimpoeas, M. (2012). Several inequalities regarding Stanley depth. Romanian Journal of Math. and Computer Science, 2(1), 28-40.
\bibitem{MC8} Cimpoeas, M. (2013). Stanley depth of squarefree Veronese ideals. Analele ştiinţifice ale Universităţii" Ovidius" Constanţa. Seria Matematică, 21(3), 67-72.
\bibitem{cycledepth} Cimpoeas, M. (2014). On the Stanley depth of edge ideals of line and cyclic graphs. arXiv preprint arXiv:1411.0624.
\bibitem{cocoa} CoCoATeam, CoCoA: A system for doing computations in commutative algebra, available at http://cocoa.dima.unige.it/.

\bibitem{DD} Davis, G. J.,  Domke, G. S. (2002). 3-circulant graphs. Journal of combinatorial mathematics and Combinatorial Computing, 40, 133-142.
\bibitem{21} 
			Duval, A. M., Goeckner, B., Klivans, C. J.,  Martin, J. L. (2016). A non-partitionable Cohen–Macaulay simplicial complex. Advances in Mathematics, 299, 381-395.

   
			\bibitem{macaulay} Grayson, D. R.,  Stillman, M. E. (2002). Macaulay2, a software system for research in algebraic geometry. available at https://www.unimelb-macaulay2.cloud.edu.au/home.
   \bibitem{herhib} Herzog, J.,  Hibi, T. (2011). Monomial ideals. Springer-Verlag  London  Limited, 3-22.
			\bibitem{herz} Herzog, J., Vladoiu, M.,  Zheng, X. (2009). How to compute the Stanley depth of a monomial ideal. Journal of Algebra, 322(9), 3151-3169.
			\bibitem{hersurvey} Herzog, J. (2013). A survey on Stanley depth. In monomial Ideals, computations and applications, A. Bigatti, P. Giménez, E. Sáenz-de-Cabezón. Proceedings of MONICA 2011, Lecture Notes in Math. 2083, Springer, Heidelberg.

\bibitem{mulnet} Hwang, F. K. (2003). A survey on multi-loop networks. Theoretical Computer Science, 299(1-3), 107-121.
\bibitem {Ichim2} Ichim, B., Katthän, L.,  Moyano-Fernández, J. J. (2017). How to compute the Stanley depth of a module. Mathematics of Computation, 86(303), 455-472.


\bibitem{MI} Ishaq, M.,  Qureshi, M. I. (2013). Upper and lower bounds for the Stanley depth of certain classes of monomial ideals and their residue class rings. Communications in Algebra, 41(3), 1107-1116.
\bibitem{ZIA} Iqbal, Z., Ishaq, M.,  Aamir, M. (2018). Depth and Stanley depth of the edge ideals of square paths and square cycles. Communications in Algebra, 46(3), 1188-1198.
\bibitem{KS} Keller, M. T., Shen, Y. H., Streib, N.,  Young, S. J. (2011). On the Stanley depth of squarefree Veronese ideals. Journal of Algebraic Combinatorics, 33, 313-324.
\bibitem{betti} Makvand, M. A.,  Mousivand, A. (2019). Betti numbers of some circulant graphs. Czechoslovak Mathematical Journal, 69(3), 593-607.
\bibitem{pathsusan} Morey, S. (2010). Depths of powers of the edge ideal of a tree. Communications in Algebra, 38(11), 4042-4055.
\bibitem{PFY} Pournaki, M., Seyed Fakhari, S. A., Yassemi, S. (2013). Stanley depth of powers of the edge ideal of a forest. Proceedings of the American Mathematical Society, 141(10), 3327-3336.
\bibitem{AR1} Rauf, A. (2010). Depth and Stanley depth of multigraded modules. Communications in Algebra, 38(2), 773-784.	

\bibitem{rinaldo} Rinaldo, G. (2018). Some algebraic invariants of edge ideal of circulant graphs. Bulletin mathématique de la Société des Sciences Mathématiques de Roumanie, 61(1), 95-105.
\bibitem{error} Sachkov, V. N., Tarakanov, V. E. (2002). Combinatorics of nonnegative matrices. In: Translations  of  Mathematical  Monographs.  Vol.  213.  Providence:  AmericanMathematical Society.
\bibitem{bakht} Shaukat, B., Haq, A. U.,  Ishaq, M. (2022). Some algebraic invariants of the residue class rings of the edge ideals of perfect semiregular trees. Communications in Algebra, 51(6), 2364-2383.
\bibitem{sdepth} Shaukat, B., Ishaq, M., Haq, A. U., Iqbal, Z. (2022). Algebraic properties of edge ideals of corona product of certain graphs. 
https://doi.org/10.48550/arXiv.2211.05721.

 
\bibitem{20} 
			Stanley, R. P. (1982). Linear Diophantine equations and local cohomology. Inventiones mathematicae, 68(2), 175-193.
			
 \bibitem{stef} Ştefan, A. (2014). Stanley depth of powers of the path ideal. arXiv preprint arXiv:1409.6072.
\bibitem{circulent} Uribe-Paczka, M. E.,  Van Tuyl, A. (2019). The regularity of some families of circulant graphs. Mathematics, 7(7), 657.

\bibitem{cohenmac} Vander Meulen, K. N., Van Tuyl, A.,  Watt, C. (2014). Cohen–Macaulay circulant graphs. Communications in Algebra, 42(5), 1896-1910.
 \bibitem{ladder} Verma, R. (2022). On some algebraic properties of edge ideals of ladder graphs. Communications in Algebra, 50(6), 2296-2311.
\bibitem{book} Villarreal, R. H.(2001). Monomial algebras. Monographs and Textbooks in Pure and Applied Mathematics. New York: Marcel Dekker, Inc., Vol. 238.
 
\end{thebibliography}
\end{document}